%% file: u-stat_v3.tex
\documentclass[preprint]{imsart}
\arxiv{arXiv:1801.05565}
\usepackage{amsmath,amsthm,verbatim,amssymb,amsfonts,amscd, graphicx}
\usepackage{amssymb,bm}
\usepackage{mathrsfs}
\usepackage{subfig}
\RequirePackage[colorlinks,citecolor=blue,urlcolor=blue]{hyperref}
\RequirePackage[numbers]{natbib}
\usepackage{constants}
\usepackage{enumerate}

\newcommand{\rom}[1]{\uppercase\expandafter{\romannumeral #1\relax}}
\newcommand{\beas}{\begin{eqnarray*}}
\newcommand{\enas}{\end{eqnarray*}}
\newcommand{\bea}{\begin{eqnarray}}
\newcommand{\ena}{\end{eqnarray}}
\newcommand{\bms}{\begin{multline*}}
\newcommand{\ems}{\end{multline*}}
\newcommand{\bels}{\begin{align*}}
\newcommand{\enls}{\end{align*}}
\newcommand{\bel}{\begin{align}}
\newcommand{\enl}{\end{align}}

\newcommand{\ignore}[1]{}
\newcommand{\tr}{\mbox{tr\,}}

\startlocaldefs

\numberwithin{equation}{section}

\newtheorem{definition}{Definition}[section]
\newtheorem{theorem}{Theorem}[section]

\newtheorem{corollary}{Corollary}[section]
\newtheorem{proposition}{Proposition}[section]
\newtheorem{lemma}{Lemma}[section]

\newtheorem{remark}{Remark}[section]

\newtheorem{fact}{Fact}

\endlocaldefs

\def\blfootnote{\xdef\@thefnmark{}\@footnotetext}

\newcommand{\expect}[1]{\mathbb{E}{\l[#1\r]}}

\input{stas}

\begin{document}

\begin{frontmatter}
\title{Robust Modifications of U-statistics and Applications to Covariance Estimation Problems}
\runauthor{S. Minsker and X. Wei}
\runtitle{Robust U-statistics}

\begin{aug}
\author{\fnms{Stanislav} \snm{Minsker}\thanksref{a,e1}\ead[label=e1,mark]{minsker@usc.edu}}
\and
\author{\fnms{Xiaohan} \snm{Wei}\thanksref{b,e2}\ead[label=e2,mark]{xiaohanw@usc.edu}}

\address[a]{Department of Mathematics, University of Southern California, Los Angeles, CA 90089.
\printead{e1}}
\address[b]{Department of Electrical Engineering, University of Southern California, Los Angeles, CA 90089.
\printead{e2}}

\end{aug}

\begin{abstract}
Let $Y$ be a $d$-dimensional random vector with unknown mean $\mu$ and covariance matrix $\Sigma$. 
This paper is motivated by the problem of designing an estimator of $\Sigma$ that admits tight deviation bounds in the operator norm under minimal assumptions on the underlying distribution, such as existence of only 4th moments of the coordinates of $Y$. 
To address this problem, we propose robust modifications of the operator-valued U-statistics, obtain non-asymptotic  guarantees for their performance, and demonstrate the implications of these results to the covariance estimation problem under various structural assumptions. 
\end{abstract}
\begin{keyword}
\kwd{U-statistics}
\kwd{heavy tails}
\kwd{covariance estimation}
\kwd{robust estimators}
\end{keyword}
\end{frontmatter}

\section{Introduction}

In mathematical statistics, it is common to assume that data satisfy an underlying model along with a set of assumptions on this model -- for example, that the sequence of vector-valued observations is i.i.d. and has multivariate normal distribution. 
Since real-world data typically do not fit the model or satisfy the assumptions exactly (e.g., due to outliers and noise), reducing the number and strictness of the assumptions helps to reduce the gap between the ``mathematical'' world and the ``real'' world. The concept of robustness occupies one the central roles in understanding this gap. 
One of the viable ways to model noisy data and outliers is to assume that the observations are generated by a heavy-tailed distribution, and this is precisely the approach that we follow in this work.  

Robust M-estimators introduced by P. Huber \cite{huber1964robust} constitute a powerful method in the toolbox for the analysis of heavy-tailed data. 
Huber noted that ``it is an empirical fact that the best [outlier] rejection procedures do not quite reach the performance of the best robust procedures.'' 
His conclusion remains valid in today's age of high-dimensional data that poses new challenging questions and demand novel methods. 

The goal of this work is to introduce robust modifications for the class of operator-valued U-statistics, which naturally appear in the problems related to estimation of covariance matrices. 
Statistical estimation in the presence of outliers and heavy-tailed data has recently attracted the attention of the research community, and the literature on the topic covers the wide range of topics. 
A comprehensive review is beyond the scope of this section, so we mention only few notable contributions. 
Several popular approach to robust covariance estimation and robust principal component analysis are discussed in \cite{hubert2008high,polyak2017principle,candes2011robust}, including the Minimum Covariance
Determinant (MCD) estimator and the Minimum Volume Ellipsoid estimator (MVE). 
Maronna's \cite{maronna1976} and Tyler's \cite{tyler1987distribution,zhang2016marvcenko} M-estimators are other well-known alternatives. Rigorous results for these estimators are available only for special families of distributions, such as elliptically symmetric. 
Robust estimators based on Kendall's tau have been recently studied in \cite{wegkamp2016adaptive,han2017statistical}, again for the family of elliptically symmetric distributions and its generalizations. 

The papers \cite{catoni2016pac,catoni2017dimension,giulini2016robust} discuss robust covariance estimation for heavy-tailed distributions and are all based on the ideas originating in work \cite{catoni2012challenging} that provided detailed non-asymptotic analysis of robust M-estimators of the univariate mean. 
The present paper can be seen as a direct extension of these ideas to the case of matrix-valued U-statistics, and continues the line of work initiated in \cite{fan2016robust} and \cite{minsker2016sub}; the main advantage of the techniques proposed is that they result in estimators that can be computed efficiently, and cover scenarios beyond covariance estimation problem. 
Recent advances in this direction include the works \cite{fan2017farm} and \cite{wei2017estimation} that present new results on robust covariance estimation; see Remark \ref{remark:comparison} for more details. 

Finally, let us mention the paper \cite{joly2016robust} that investigates robust analogues of U-statistics obtained via the median-of-means technique \cite{alon1996space,devroye2016sub,Nemirovski1983Problem-complex00,lerasle2011robust}. 
We include a more detailed discussion and comparison with the methods of this work in Section \ref{section:main} below. 

The rest of the paper is organizes as follows. 
Section \ref{sec:prelim} explains the main notation and background material. Section \ref{section:main} introduces the main results. 
Implications for covariance estimation problem and its versions are outlined in Section \ref{sec:covariance}. 
Finally, the proofs of the main results are contained in Section \ref{section:proofs}.


\section{Preliminaries}
\label{sec:prelim}

In this section, we introduce main notation and recall useful facts that we rely on in the subsequent exposition.

\subsection{Definitions and notation}
\label{sec:definitions}

Given $A\in \mb C^{d_1\times d_2}$, let $A^\ast\in \mb C^{d_2\times d_1}$ be the Hermitian adjoint of $A$. 
The set of all $d\times d$ self-adjoint matrices will be denoted by $\mb H^d$. 
For a self-adjoint matrix $A$, we will write $\lambda_{\mx}(A)$ and $\lambda_{\mn}(A)$ for the largest and smallest eigenvalues of $A$. Hadamard (entry-wise) product of matrices $A,B\in \mb C^{d_1\times d_2}$ will be denoted $A_1\odot A_2$. 
Next, we will introduce the matrix norms used in the paper. 

Everywhere below, $\|\cdot\|$ stands for the operator norm $\|A\|:=\sqrt{\lambda_{\mx}(A^\ast A)}$. 
If $d_1=d_2=d$, we denote by $\tr A$ the trace of $A$.
Next, for $A\in \mb C^{d_1\times d_2}$, the nuclear norm $\|\cdot\|_1$ is defined as 
$\|A\|_1=\tr(\sqrt{A^*A})$, where $\sqrt{A^*A}$ is a nonnegative definite matrix such that $(\sqrt{A^*A})^2=A^\ast A$. 
The Frobenius (or Hilbert-Schmidt) norm is $\|A\|_{\mathrm{F}}=\sqrt{\tr(A^\ast A)}$, and the associated inner product is 
$\dotp{A_1}{A_2}=\tr(A_1^\ast A_2)$. 
Finally, define $\|A\|_{\max}:=\sup_{i,j}|A_{i,j}|$. 
For a vector $Y\in \mb R^d$, $\l\| Y \r\|_2$ stands for the usual Euclidean norm of $Y$. 

\noindent Given two self-adjoint matrices $A$ and $B$, we will write $A\succeq B \ (\text{or }A\succ B)$ iff $A-B$ is nonnegative (or positive) definite.

\noindent Given a random matrix $Y\in \mb C^{d_1\times d_2}$ with $\mb E\|Y\|<\infty$, the expectation $\mb EY$ denotes a $d_1\times d_2$ matrix such that $\l( \mb EY \r)_{i,j} = \mb EY_{i,j}$. For a sequence $Y_1,\ldots,Y_n$ of random matrices, $\mb E_j[\,\cdot \,]$ will stand for the conditional expectation $\mb E[\,\cdot\,|Y_1,\ldots,Y_{j}]$. 

\noindent For $a,b\in \mb R$, set $a\vee b:=\max(a,b)$ and $a\wedge b:=\min(a,b)$. Finally, recall the definition of the function of a matrix-valued argument. 
\begin{definition}
\label{matrix-function}
Given a real-valued function $f$ defined on an interval $\mb T\subseteq \mb R$ and a self-adjoint $A\in \mb H^d$ with the eigenvalue decomposition 
$A=U\Lambda U^\ast$ such that $\lambda_j(A)\in \mb T,\ j=1,\ldots,d$, define $f(A)$ as 
$f(A)=Uf(\Lambda) U^\ast$, where 
\[
f(\Lambda)=f\l( \begin{pmatrix}
\lambda_1 & \,  & \,\\
\, & \ddots & \, \\
\, & \, & \lambda_d
\end{pmatrix} \r)
=\begin{pmatrix}
f(\lambda_1) & \,  & \,\\
\, & \ddots & \, \\
\, & \, & f(\lambda_d)
\end{pmatrix}.
\] 
\end{definition}

Finally, we introduce the Hermitian dilation which allows to reduce the problems involving general rectangular matrices to the case of Hermitian matrices. 
\begin{definition}
Given the rectangular matrix $A\in\mb C^{d_1\times d_2}$, the Hermitian dilation $\m D: \mb C^{d_1\times d_2}\mapsto \mb C^{(d_1+d_2)\times (d_1+d_2)}$ is defined as
\begin{align}
\label{eq:dilation}
&
\m D(A)=\begin{pmatrix}
0 & A \\
A^\ast & 0
\end{pmatrix}.
\end{align}
\end{definition}
Since 
$\m D(A)^2=\begin{pmatrix}
A A^\ast & 0 \\
0 & A^\ast A
\end{pmatrix},$ 
it is easy to see that $\| \m D(A) \|=\|A\|$. 

\subsection{U-statistics}
\label{sec:setup}

Consider a sequence of i.i.d. random variables $X_1,\ldots,X_n$ ($n\geq2$) taking values in a measurable space $(\m S,\mathcal{B})$, and let $P$ be the distribution of $X_1$.  
Assume that $H: \m S^m\rightarrow\mb H^d$ ($2\leq m\leq n$) is a $\mathcal{S}^m$-measurable permutation symmetric kernel, meaning that $H(x_1,\ldots,x_m) = H(x_{\pi_1},\ldots,x_{\pi_m})$ for any $(x_1,\ldots,x_m)\in \m S^m$ and any permutation $\pi$. 
The U-statistic with kernel $H$ is defined as \cite{hoeffding1948class}
\begin{equation}
\label{u-stat}
U_n:=\frac{(n-m)!}{n!}\sum_{(i_1,\ldots,i_m)\in I_n^m}H(X_{i_1},\ldots,X_{i_m}),
\end{equation} 
where $I_n^m:=\{(i_1,\ldots,i_m):~1\leq i_j\leq n,,~i_j\neq i_k~\textrm{if}~j\neq k\}$; clearly, it is an unbiased estimator of $\mb EH(X_1,\ldots,X_m)$. Throughout this paper, we will impose a mild assumption stating that $\mb E\l\| H(X_1,\ldots,X_m)^2 \r\|<\infty$. 

One of the key questions in statistical applications is to understand the concentration of a given estimator around the unknown parameter of interest. Majority of existing results for U-statistics assume that the kernel $H$ is bounded \cite{arcones1993limit}, or that $\l\| \mb EH(X_1,\ldots,X_m)\r\|$ has sub-Gaussian tails \cite{gine2000exponential}. 
However, in the case when only the moments of low orders of $\l\| H(X_1,\ldots,X_m)\r\|$ are finite, deviations of the random variable 
\[
\l\| H(X_1,\ldots,X_m) - \mb EH(X_1,\ldots,X_m)\r\|
\] 
do not satisfy exponential concentration inequalities. 
At the same time, as we show in this paper, it is possible to construct ``robust modifications'' of $U_n$ for which sub-Gaussian type deviation results hold.   
 
In the remainder of this section, we recall several useful facts about U-statistics.  
The projection operator $\pi_{m,k}~(k\leq m)$ is defined as 
\[
\pi_{m,k}H(\mathbf{x}_{i_1},\ldots,\mathbf{x}_{i_k})
:= (\delta_{\mathbf{x}_{i_1}}- P)\ldots(\delta_{\mathbf{x}_{i_k}} - P) P^{m-k}H,
\]
where 
\[
\mathcal{Q}^mH := \int\ldots\int H(\mathbf{y}_1,\ldots,\mathbf{y}_m)dQ(\mathbf{y}_1)\ldots dQ(\mathbf{y}_m),
\] 
for any probability measure $Q$ in $(\m S,\mathcal{B})$, and $\delta_{x}$ is a Dirac measure concentrated at $x\in \m S$. 
For example, $\pi_{m,1}H(x) = \mb E \l[ H(X_1,\ldots,X_m)| X_1=x\r] - \mb E H(X_1,\ldots,X_m)$.
\begin{definition}
An $\mathcal{S}^m$-measurable function $F: \m S^m\rightarrow\mb H^d$ is $P$-degenerate of order $r$ 
($1\leq r<m$), if
\[
\mb EF(\mathbf{x}_1,\ldots,\mathbf{x}_r,X_{r+1},\ldots,X_m)=0,~\forall \mathbf{x}_1,\ldots,\mathbf{x}_r\in \m S,
\]
and $\mb E F(\mathbf{x}_1,\ldots,\mathbf{x}_r, \mathbf{x}_{r+1},X_{r+2},\ldots,X_m)$ is not a constant function. 
Otherwise, $F$ is non-degenerate.
\end{definition}
The following result is commonly referred to as Hoeffding's decomposition; see \cite{Decoupling} for details. 
\begin{proposition}
\label{hoeffding}
The following equality holds almost surely:
\begin{equation*}
U_n=\sum_{k=0}^m{m \choose k}V_n(\pi_{m,k}H),
\end{equation*}
where
\[
V_n(\pi_{m,k}H)=\frac{(n-k)!}{n!}\sum_{(i_1,\ldots,i_k)\in I^k_n}\pi_{m,k}H(X_{i_1},\ldots,X_{i_k}).
\]
\end{proposition}
\noindent For instance, the first order term ($k=1$) in the decomposition is 
\[
m V_n(\pi_{m,1}H)=\frac{m}{n}\sum_{j=1}^n\pi_{m,1}H(X_j).
\]
In this paper, we consider non-degenerate U-statistics which commonly appear in applications such as estimation of covariance matrices and that serve as a main motivation for this paper. It is well-known that 
\[
\mb E \l(U_n- \mb E H(X_1,\ldots,X_m)\r)^2
={n \choose m}^{-1}\sum_{k=1}^m{m \choose k}{n-m \choose m-k}\Sigma_k^2,
\]
where $\Sigma_k^2 =\mb E\big(\pi_{m,k}H(X_1,X_2,\ldots,X_k)\big)^2$, $k=1,\ldots,m$.
As $n$ gets large, the first term in the sum above dominates the rest that are of smaller order, so that 
\[
\left\|\expect{(U_n-\mathcal{P}^mH)^2}\right\|
= \left\|{n \choose m}^{-1}m{n-m \choose m-1}\Sigma_1^2\right\| + o(n^{-1})
= \Big\|\frac{m^2}{n}\Sigma_1^2 \Big\| + o(n^{-1})
\]
as $n\to\infty$.

\section{Robust modifications of U-statistics}
\label{section:main}

The goal of this section is to introduce the robust versions of U-statistics, and state the main results about their performance. 
\noindent Define 
\begin{equation}
\label{eq:psi}
\psi(x) = 
\begin{cases}
1/2, & x >1,\\
x - \sign(x)\cdot x^2/2, & |x|\leq 1, \\
-1/2, & x<-1
\end{cases}
\end{equation}
and its antiderivative 
\begin{equation}
\label{eq:psi2}
\Psi(x) = 
\begin{cases}
\frac{x^2}{2} - \frac{|x|^3}{6}, & |x| \leq 1,\\
\frac{1}{3} + \frac{1}{2}(|x|-1),  & |x|>1.
\end{cases}
\end{equation}
The function $\Psi(x)$ is closely related to Huber's loss \cite{huber2011robust}; concrete choice of $\Psi(x)$ is motivated by its properties, namely convexity and the fact that its derivative $\psi(x)$ is operator Lipschitz and bounded (see Lemma \ref{lemma:optimization} below). 
\begin{figure}[t]
 \centering
  \subfloat[$\psi(x)$]{
    \boxed{\includegraphics[width=0.5\textwidth]{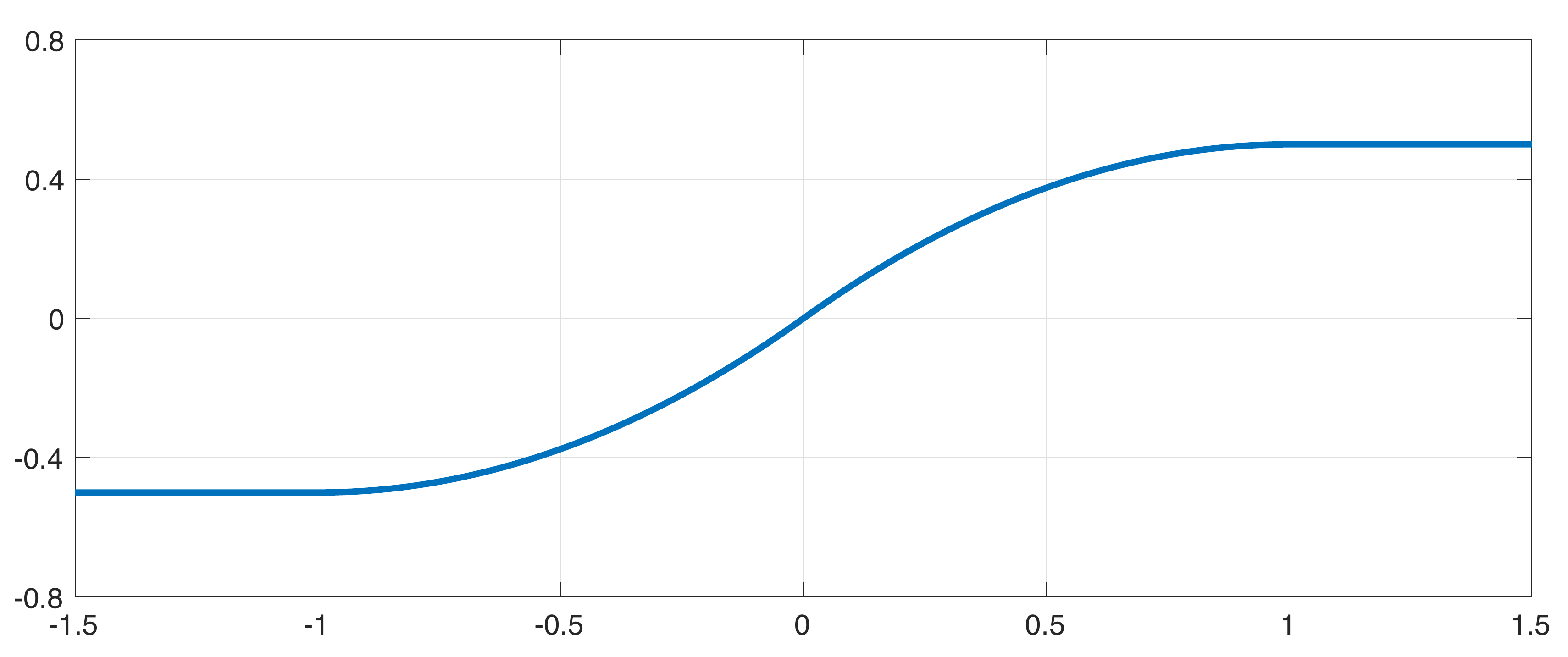}}
    \label{psi1}}
  \subfloat[$\Psi(x)$]{
   \boxed{ \includegraphics[width=0.5\textwidth]{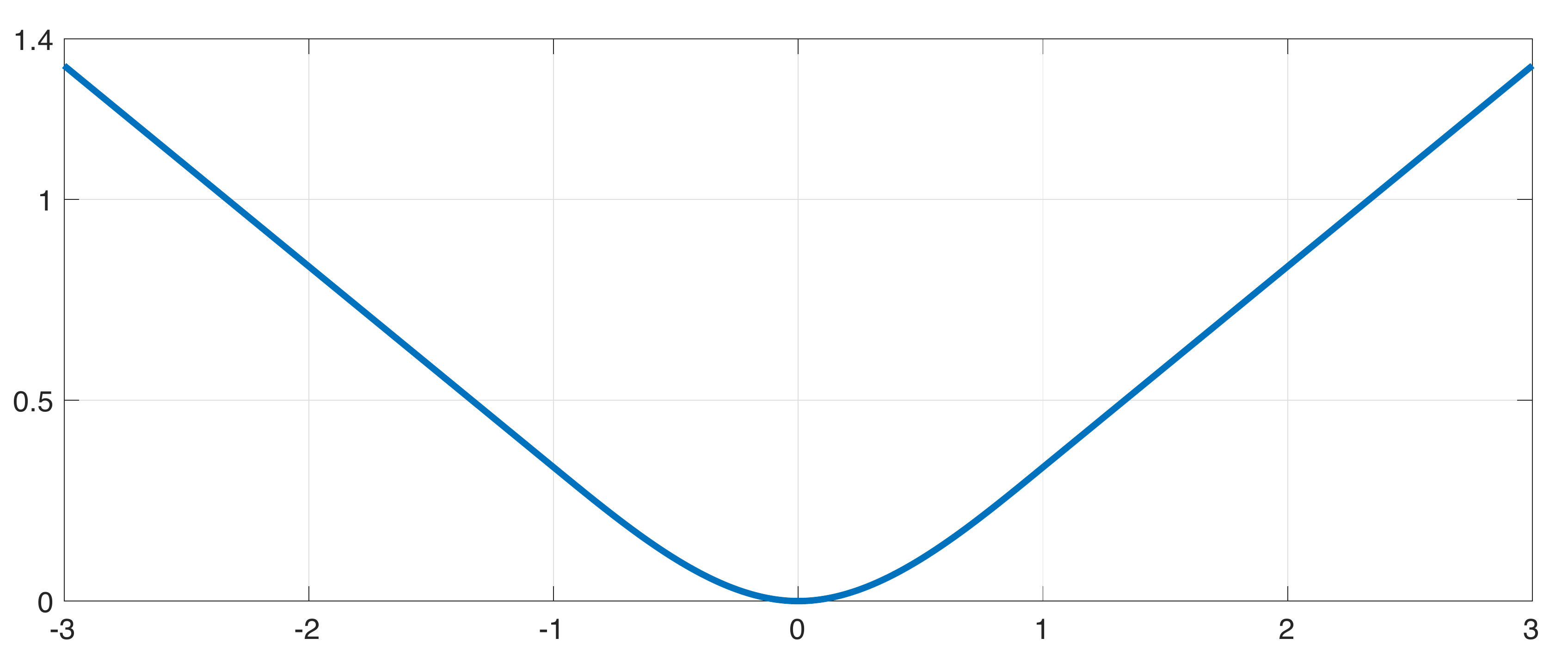}}
    \label{psi2}}
  \caption{Graphs of the functions $\psi(x)$ and $\Psi(x)$.}
\end{figure}
Let $U_n$ be $\mb H^d$-valued U-statistic,
\[
U_n:=\frac{(n-m)!}{n!}\sum_{(i_1,\ldots,i_m)\in I_n^m} H(X_{i_1},\ldots,X_{i_m}).
\]
Since $U_n$ is the average of matrices of the form $H(X_{i_1},\ldots,X_{i_m}), \ (i_1,\ldots,i_m)\in I_n^m,$ it can be equivalently written as 
\begin{align*}
U_n & = \argmin_{U\in \mb H^d} 
\sum_{(i_1,\ldots,i_m)\in I_n^m} \left\| H(X_{i_1},\ldots,X_{i_m}) - U\right\|^2_{\mathrm{F}} \\
& 
=\argmin_{U\in \mb H^d} \tr\Bigg[ \sum_{(i_1,\ldots,i_m)\in I_n^m} 
\left( H(X_{i_1},\ldots,X_{i_m}) - U \right)^2 \Bigg].
\end{align*}
A robust version of $U_n$ is then defined by replacing the quadratic loss by (rescaled) loss $\Psi(x)$. 
Namely, let $\theta>0$ be a scaling parameter, and define 
\begin{align}
\label{eq:estimator1}
\wh U_n^\star & = 
\argmin_{U\in \mb H^d} \tr\Bigg[ \sum_{(i_1,\ldots,i_m)\in I_n^m} 
\Psi\Big( \theta\left( H(X_{i_1},\ldots,X_{i_m}) - U\right)\Big) \Bigg].
\end{align}
For brevity, we will set 
\[
H_{i_1\ldots i_m}:= H(X_{i_1},\ldots,X_{i_m}) \text{ and } \mb EH:=\mb EH_{i_1\ldots i_m}
\]
in what follows. 
Define 
\begin{align}
\label{eq:F}
&
F_\theta(U):= \frac{1}{\theta^2}\frac{(n-m)!}{n!}\tr\Bigg[ \sum_{(i_1,\ldots,i_m)\in I_n^m} 
\Psi\Big( \theta\left( H_{i_1\ldots i_m} - U\right)\Big) \Bigg].
\end{align}
Clearly, $\wh U_n^\star$ can be equivalently written as 
\[
\wh U_n^\star = \argmin_{U\in \mb H^d} \tr \l[ F_\theta(U) \r].
\]
The following result describes the basic properties of this optimization problem.
\begin{lemma}
\label{lemma:optimization}
The following statements hold:
\begin{enumerate}
\item Problem \eqref{eq:estimator1} is a convex optimization problem.
\item The gradient $\nabla F_\theta(U)$ can be represented as 
\[
\nabla F_\theta(U) = -\frac{1}{\theta}\frac{(n-m)!}{n!}\sum_{(i_1,\ldots,i_m)\in I_n^m} 
\psi\Big(\theta\l( H_{i_1\ldots i_m} - U\r) \Big).
\] 
Moreover, $\nabla F_\theta(\cdot): \mb H^d\mapsto \mb H^d$ is Lipschitz continuous in Frobenius and operator norms with Lipschitz constant $1$. 
\item Problem \eqref{eq:estimator1} is equivalent to 
\begin{align}
\label{eq:estimator3}
\sum_{(i_1,\ldots,i_m)\in I^m_n}\psi\Big( \theta \left(H_{i_1\ldots i_m} - \wh U_n^\star\right) \Big) = 0_{d\times d}.
\end{align}
\end{enumerate}
\end{lemma}
Proofs of these facts are given in Section \ref{proof:opt}. Next, we present our main result regarding the performance of the estimator $\wh U_n^\star$. 
Define the \textit{effective rank} \cite{vershynin2010introduction} of a nonnegative definite matrix $A\in \mb H^d$ as
\[
\mathrm{r}(A) = \frac{\tr A}{\| A \|}.
\]
It is easy to see that for any matrix $A\in \mb H^d$, $\mathrm{r}(A)\leq d$. 
We will be interested in the effective rank of the matrix $\mb E \l( H_{1\ldots m} - \mb EH \r)^2$, and will denote 
\[
\mathrm{r}_H:= \mathrm{r}\l( \mb E \l( H_{1\ldots m} - \mb EH \r)^2\r).
\]
\begin{theorem}
\label{thm:new-performance}
Let $k=\lfloor n/m \rfloor$, and assume that $t>0$ is such that 
\[
\mathrm{r}_H\frac{ t}{k}\leq \frac{1}{104}.
\]
Then for any $\sigma\geq \| \mb E \l( H_{1\ldots m} - \mb EH \r)^2 \|^{1/2}$ and 
$\theta:=\theta_\sigma=\frac{1}{\sigma}\sqrt{\frac{2t}{k}}$, 
\[
\l\| \wh U_n^\star - \mb EH\r\| \leq 23\sigma\sqrt{\frac{t}{k}}
\]
with probability $\geq 1-(4d+1)e^{-t}$.
\end{theorem}
\noindent The proof is presented in Section \ref{proof:new-performance}. 
\begin{remark}
\label{remark:rank}
Condition $\mathrm{r}_H\frac{ t}{k}\leq \frac{1}{104}$ in Theorem \ref{thm:new-performance} can be weakened to 
\[
\frac{\tr\l( \mb E \l( H_{1\ldots m} - \mb EH \r)^2 \r)}{\sigma^2} \frac{t}{k}\leq \frac{1}{104},
\]
where $\sigma^2\geq \| \mb E \l( H_{1\ldots m} - \mb EH \r)^2 \|$. This fact follows from the straightforward modification of the proof of Theorem \ref{thm:new-performance} and can be useful in applications.
\end{remark}

\begin{remark}
\label{remark:mom}
The paper \cite{joly2016robust} investigates robust analogues of univariate U-statistics based on the median-of-means  (MOM) technique. 
This approach can be extended to higher dimensions via replacing the univariate median by an appropriate multivariate generalization (e.g., the spatial median). 
When applied to covariance estimation problem, it yields estimates for the error measured in Frobenius norm; however, is not not clear whether it can be used to obtain the error bounds in the operator norm. 
More specifically, to obtain such a bound via the MOM method, one would need to estimate 
$
\mb E \l\| \frac{1}{n}\sum_{j=1}^n (Y_j - \mb EY)(Y_j - \mb EY)^T - \Sigma \r\|^2,
$
where $Y_1,\ldots,Y_j$ are i.i.d. copies of a random vector $Y\in \mb R^d$ such that $\mb E(Y-\mb EY)(Y-\mb EY)^T=\Sigma$ and $\mb E\|Y\|_2^4<\infty$. 
We are not aware of any existing (non-trivial) upper bounds for the aforementioned expectation that require only 4 finite moments of $\|Y\|_2$. 
On the other hand, it is straightforward to obtain the upper bound in the Frobenius norm as 
$
\mb E \big\| \frac{1}{n}\sum_{j=1}^n (Y_j - \mb EY)(Y_j - \mb EY)^T - \Sigma \big\|_{\mathrm{F}}^2 = 
\frac{1}{n}\l(\mb E\| Y - \mb EY\|_2^4  - \l\| \Sigma \r\|^2_{\mathrm{F}}\r).
$
\end{remark}


\subsection{Construction of the adaptive estimator}
\label{section:lepski}

The downside of the estimator $\wh U_n^\star$ defined in \eqref{eq:estimator1} is the fact that it is not completely data-dependent as the choice of $\theta$ requires the knowledge of an upper bound on 
\[
\sigma_\ast^2 := \l\| \mb E \l( H_{1\ldots m} - \mb EH \r)^2\r\|.
\] 
To alleviate this difficulty, we propose an adaptive construction based on a variant of Lepski's method \cite{lepskii1992asymptotically}.

Assume that $\sigma_{\mn}$ is a known (possible crude) lower bound on $\sigma_\ast$. 
Choose $\gamma>1$, let $\sigma_j := \sigma_{\mn} \gamma^j$, and
for each integer $j\geq 0$, set $t_j:=t+\log \l[j(j+1)\r]$ and 
\[
\theta_j=\theta(j,t)=\sqrt{\frac{2 t_j}{k}}\frac{1}{\sigma_j},
\] 
where $k=\lfloor n/m \rfloor$ as before. 
Let
\[
\wh U_{n,j}=\argmin_{U\in \mb H^d} F_{\theta_j}(U),
\]
with $F_{\theta}$ was defined in \eqref{eq:F}.
Finally, set 
\[
\m L:=\m L(t) = \l\{ l\in \mb N: \ \mathrm{r}_H\frac{t_l}{k} \leq \frac{1}{104}\r\}
\] 
and
\begin{align}
\label{eq:lepski}
j_\ast:=\min\l\{ j\in \m L: \forall l\in \m L, \ l>j ,\ \l\|  \wh U_{n,l} - \wh U_{n,j} \r\|\leq 46\sigma_{l} \sqrt{\frac{ t_l}{k}}  \r\}
\end{align}
and $\widetilde U_n^\star:=\wh U_{n,j_\ast}$; if condition \eqref{eq:lepski} is not satisfied by any $j\in \m L$, we set $j_\ast=+\infty$ and $\widetilde U_n^\star=0_{d\times d}$.
\noindent Let 
\begin{align}
\label{eq:xi}
&
\Xi = \log\l[\l( \Big\lfloor \frac{\log \l(\sigma_\ast/\sigma_{\mn}\r)}{\log \gamma}\Big\rfloor+1\r)\l(\Big\lfloor \frac{\log \l(\sigma_\ast/\sigma_{\mn}\r)}{\log \gamma}\Big\rfloor+2 \r)\r].
\end{align}
\begin{theorem}
\label{th:lepski}
Assume that $t>0$ is such that 
\[
\mathrm{r}_H\frac{(t+\Xi)}{k}\leq \frac{1}{104}.
\]
Then with probability $\geq 1 - (4d+1)e^{-t}$,
\[
\l\| \widetilde U_n^\star - \mb EH \r\| \leq
69\gamma\cdot\sigma_\ast \sqrt{\frac{t+\Xi}{k}},
\] 
\end{theorem}
\noindent In other words, adaptive estimator can be obtained at the cost of the additional multiplicative factor $3\gamma$ in the error bound. 
\begin{proof}
Let $\bar j=\min\l\{  j\geq 1: \ \sigma_j \geq \sigma_\ast\r\}$, and note that 
$\bar j\leq \Big\lfloor \frac{\log \l(\sigma_\ast/\sigma_{\mn}\r)}{\log \gamma}\Big\rfloor+1$ and $\sigma_{\bar j}\leq \gamma \sigma_\ast$. 
Note that condition of Theorem \ref{th:lepski} guarantees that $\bar j \in \m L$. 
We will show that $j_\ast \leq \bar j$ with high probability. 
Indeed,
\begin{align*}
\Pr\l( j_\ast > \r. & \l. \bar j\r)\leq \Pr\l( \bigcup_{l\in \m L: l > \bar j} \l\{  \l\|  \wh U_{n,l} - \wh U_{n,\bar j} \r\| > 46\sigma_{l} \sqrt{\frac{t_j}{k}} \r\} \r)\\
& 
\leq \Pr\l(  \l\|  \wh U_{n,\bar j}  - \mb EH \r\| > 23\sigma_{\bar j} \sqrt{\frac{t_{\bar j}}{k}} \r) + 
\sum_{l\in \m L: l>\bar j}\Pr\l( \l\|   \wh U_{n,l}  - \mb EH \r\| > 23\sigma_{l} \sqrt{\frac{t_l}{k}}  \r)   \\
&
\leq (4d+1) e^{-t}\frac{1}{\bar j(\bar j+1)} + (4d+1) e^{-t} \sum_{l>\bar j}\frac{1}{l(l+1)}\leq (4d+1) e^{-t}.
\end{align*}
where we used Theorem \ref{thm:new-performance} to bound each of the probabilities in the sum. 
The display above implies that the event 
\[
\m B = \bigcap_{l\in \m L: l\geq \bar j} 
\l\{ \l\|  \wh U_{n,l} - \mb EH \r\|\leq 23 \sigma_{l} \sqrt{\frac{t_l}{k}}  \r\} 
\]
of probability $\geq 1-(4d+1) e^{-t}$ is contained in $\m E=\l\{  j_\ast\leq \bar j \r\}$. 
Hence, on $\m B$ we have  
\begin{align*}
\l\| \widetilde U_n^\star - \mb EH \r\|&
\leq \| \widetilde U_n^\star - \wh U_{n,\bar j} \| + \| \wh U_{n,\bar j} - \mb EH \| \leq 
46 \sigma_{\bar j} \sqrt{\frac{t_{\bar j}}{k}} + 23 \sigma_{\bar j} \sqrt{\frac{t_{\bar j}}{k}} \\
&\leq \gamma\cdot 69 \sigma_\ast \sqrt{\frac{t+\Xi}{k}},
\end{align*}
where $\Xi = \log\l[\l( \Big\lfloor \frac{\log \l(\sigma_\ast/\sigma_{\mn}\r)}{\log \gamma}\Big\rfloor+1\r)\l(\Big\lfloor \frac{\log \l(\sigma_\ast/\sigma_{\mn}\r)}{\log \gamma}\Big\rfloor+2 \r)\r]$.
\end{proof}

\subsection{Extension to rectangular matrices}
\label{sec:rectangular}

In this section, we assume a more general setting where $H: \m S^m\mapsto \mb C^{d_1\times d_2}$ is a $\mb C^{d_1\times d_2}$-valued permutation-symmetric function. 
As before, our goal is to construct an estimator of $\mb EH$. 
We reduce this general problem to the case of $\mb H^{d_1+d_2}$-valued functions via the self-adjoint dilation defined in \eqref{eq:dilation}. 
Let 
\[
\m D(H_{i_1\ldots i_m}) = 
\begin{pmatrix}
0 & H(X_{i_1},\ldots,X_{i_m}) \\
\l[H(X_{i_1},\ldots,X_{i_m})\r]^\ast & 0
\end{pmatrix},
\] 
and
\[
\bar U_n^\star = 
\argmin_{U\in \mb H^{d_1+d_2}} \tr\Bigg[ \sum_{(i_1,\ldots,i_m)\in I_n^m} 
\Psi\Big( \theta\left( \m D(H_{i_1\ldots i_m}) - U\right)\Big) \Bigg].
\]
Let $\hat U^\star_{11}\in \mb C^{d_1\times d_1}$, $\hat U^\star_{22}\in \mb C^{d_2\times d_2}$, $\hat U^\star_{12}\in \mb C^{d_1\times d_2}$ be such that $\bar U^\star_n$ can be written in the block form as
$\bar U_n^\star
=\begin{pmatrix}
\hat U^\star_{11} & \hat U^\star_{12} \\
( \hat U^\star_{12} )^\ast & \hat U^\star_{22}
\end{pmatrix}.$ 
Moreover, define 
\[
\sigma^2_\star :=  \max\l( \big\| \mb E (H_{1\ldots m} - \mb EH) (H_{1\ldots m} - \mb EH)^\ast \big\|, \big\|  
\mb E (H_{1\ldots m} - \mb EH)^\ast (H_{1\ldots m} - \mb EH) \big\| \r) 
\]
and 
\[
\tilde{\mathrm{r}}_H:= 2\cdot\frac{\tr\l[ \mb E (H_{1\ldots m} - \mb EH) (H_{1\ldots m} - \mb EH)^\ast\r] }{\sigma^2_\star}.
\]
\begin{corollary}
\label{cor:rectangular}
Let $k=\lfloor n/m \rfloor$, and assume that $t>0$ is such that
\[
\tilde{\mathrm{r}}_H \frac{t}{k} \leq \frac{1}{104}.
\]
Then for any $\sigma \geq \sigma_\star$ and $\theta:=\theta_\sigma=\frac{1}{\sigma}\sqrt{\frac{2t}{k}}$, 
\[
\l\| \hat U^\star_{12} - \mb EH  \r\| 
\leq 23\sigma \sqrt{\frac{t}{k}}
\]
with probability $\geq 1 - \l( 4(d_1+d_2)+1 \r)e^{-t}$.
\end{corollary}
The proof is outlined in Section \ref{proof:rectangular}.

\subsection{Computational considerations}
\label{sec:computational}

Since the estimator $\wh U_n^\star$ is the solution of the convex optimization problem \eqref{eq:estimator1}, it can be approximated via the gradient descent. 
We consider the simplest gradient descent scheme with constant step size equal $1$. 
Note that the Lipschitz constant of $F_\theta(U)$ is $L_F=1$ by Lemma \ref{lemma:optimization}, hence this step choice is exactly equal to $\frac{1}{L_F}$. 
Given a starting point $U_0\in \mb H^d$, the gradient descent iteration for minimization of $\tr F_\theta(U)$ is 
\begin{align*}
U^{(0)}_n :&= U_0, \\
U^{(j)}_n :&= U^{(j-1)}_n - \nabla \l( \tr F_\theta\l( U^{(j-1)}_n \r)\r) \\
&
=U^{(j-1)}_n+\frac{1}{\theta}\frac{(n-m)!}{n!}\sum_{(i_1,\ldots,i_m)\in I^m_n}\psi\Big(\theta\l( H_{i_1\ldots i_m} - U^{(j-1)}_n\r) \Big), \ j\geq 1.
\end{align*}
\begin{lemma}
\label{lemma:grad-descent}
The following inequalities hold for all $j\geq 1$:
\[
(a) \quad \tr \l[ F_\theta \l(U_n^{(j)}\r) - F_\theta \l(\wh U_n^\star\r)\r] \leq \frac{\l\| U_0 - \wh U_n^\star \r\|_F^2}{2j};
\]
Moreover, under the assumptions of Theorem \ref{thm:new-performance}, 
\[
(b) \quad \Big\| U_n^{(j)} - \mb EH \Big\| \leq \l( \frac{3}{4} \r)^{j}\l\| U_0 - \mb EH \r\| + 23\sigma\sqrt{\frac{t}{k}}.
\]
\end{lemma}
The proof is given is Section \ref{proof:grad-descent}. 
Note that part (b) implies that a small number of iterations suffice to get an estimator of $\mb EH$ that achieves performance bound similar to $\wh U_n^\star$.


\section{Estimation of covariance matrices}
\label{sec:covariance}

In this section, we consider applications of the previously discussed results to covariance estimation problems. 
Let $Y\in \mb R^d$ be a random vector with mean $\mb EY=\mu$, covariance matrix $\Sigma=\mb E\l[ (Y-\mu)(Y-\mu)^T \r]$, and such that $\mb E \|Y-\mu\|_2^4<\infty$.  
Assume that $Y_1,\ldots, Y_{n}$ be i.i.d. copies of $Y$. 
Our goal is to estimate $\Sigma$; note that when the observations are the heavy-tailed, mean estimation problem becomes non-trivial, so the assumption $\mu=0$ is not plausible.  

$U$-statistics offer a convenient way to avoid explicit mean estimation. 
Indeed, observe that $\Sigma = \frac{1}{2}\mb E\left[ (Y_1 - Y_2)(Y_1 - Y_2)^T\right]$, hence  
the natural estimator of $\Sigma$ is the $U$-statistic
\begin{equation}\label{eq:sample-covariance}
\widetilde \Sigma_n =  \frac{1}{n(n-1)} \sum_{i\ne j}\frac{ (Y_i - Y_j)(Y_i - Y_j)^T}{2}.
\end{equation}
It is easy to check that $\widetilde \Sigma$ coincides with the usual sample covariance estimator 
\[
\widetilde \Sigma_n=\frac{1}{n-1}\sum_{j=1}^n (Y_j - \bar Y_n)(Y_j - \bar Y_n)^T.
\]
The robust version is defined according to \eqref{eq:estimator1} as 
\begin{align}
\label{eq:estimator4}
\wh \Sigma_\star =
\argmin_{S\in \mb R^{d\times d}, S=S^T} \Bigg[ \tr \sum_{i\ne j}
\Psi\left(\theta\left( \frac{ (Y_i - Y_j)(Y_i - Y_j)^T}{2} - S\right)\right) \Bigg],
\end{align}
which, by Lemma \ref{lemma:optimization}, is equivalent to 
\begin{align*}
\sum_{i\ne j}\psi\left( \theta \left( \frac{(Y_i - Y_j)(Y_i - Y_j)^T}{2} - \wh\Sigma_\star \right) \right) = 0_{d\times d}.
\end{align*}

\begin{remark}
\label{remark:comparison}
Assume that $\Sigma_n^{(0)} = 0_{d\times d}$, then the first iteration of the gradient descent for the problem \eqref{eq:estimator4} is
\begin{align*}
\Sigma_n^{(1)} & = 
\frac{1}{\theta}\frac{1}{n(n-1)}\sum_{i\ne j}\psi\Big(\theta \, \frac{(Y_i - Y_j)(Y_i - Y_j)^T}{2} \Big). 
\end{align*}
$\Sigma_n^{(1)}$ can itself be viewed as an estimator of the covariance matrix. 
It has been proposed in \cite{minsker2016sub} (see Remark 7 in that paper), and its performance has been later analyzed in \cite{fan2017farm} (see Theorem 3.2). These results support the claim that a small number of gradient descent steps for problem \eqref{eq:estimator1} suffice in applications.
\end{remark}

\noindent To assess performance of $\wh \Sigma_\star$, we will apply Theorem \ref{thm:new-performance}. 
First, let us discuss the ``matrix variance'' $\sigma^2$ appearing in the statement. 
Direct computation shows that for $H(Y_1,Y_2) = \frac{(Y_1 - Y_2)(Y_1 - Y_2)^T}{2}$, 
\[
\mb E(H_{i_1,\ldots,i_m} - \mb EH)^2 = \frac{1}{2}\l( \mb E\l(  (Y-\mu)(Y-\mu)^T \r)^2 + \tr(\Sigma)\Sigma\r). 
\]
The following result (which is an extension of Lemma 2.3 in \cite{wei2017estimation}) connects $\big\| \mb E(H - \mb EH)^2 \big\|$ with $r(\Sigma)$, the effective rank of the covariance matrix $\Sigma$.
\begin{lemma}
\label{lemma:variance}
\begin{enumerate}
\item[(a)] Assume that kurtosis of the linear forms $\langle Y,v \rangle$ is uniformly bounded by $K$, meaning that 
$\sup\limits_{v:\|v\|_2=1}\frac{\mb E \langle Y - \mb EY, v \rangle^4}{\l[\mb E \langle Y - \mb EY, v \rangle^2\r]^2}\leq K$. Then
\begin{align*}
\big\| \mb E\l(  (Y-\mu)(Y-\mu)^T \r)^2 \big\| &\leq K \,\tr(\Sigma) \, \|\Sigma\|  
\\
\end{align*}
\item[(b)] Assume that the kurtosis of the coordinates $Y^{(j)}:=\dotp{Y}{e_j}$ of $Y$ is uniformly bounded by $K'<\infty$, meaning that 
$\max\limits_{j=1,\ldots,d}\frac{\mb E\l( Y^{(j)} - \mb E Y^{(j)} \r)^4}{\l[\mb E\l( Y^{(j)} - \mb E Y^{(j)} \r)^2\r]^2}\leq K'$. Then
\begin{align*}
&
\tr\l[  \mb E\l(  (Y-\mu)(Y-\mu)^T \r)^2 \r] \leq K' \l( \tr(\Sigma) \r)^2.
\end{align*}
\item[(c)] The following inequality holds:
\[
\big\| \mb E\l(  (Y-\mu)(Y-\mu)^T \r)^2 \big\| \geq \tr(\Sigma) \l\| \Sigma \r\|. 
\]
\end{enumerate}
\end{lemma}
\noindent Lemma \ref{lemma:variance} immediately implies that under the bounded kurtosis assumption, 
\[
\big\| \mb E(H - \mb EH)^2\big\| \leq K \,\mathrm{r}(\Sigma) \, \|\Sigma\|^2.
\]
The following corollary of Theorem \ref{thm:new-performance} (together with Remark \ref{remark:rank}) is immediate:
\begin{corollary}
\label{corollary:cov}
Assume that the kurtosis of linear forms $\langle Y,v \rangle, \ v\in \mb R^d,$ is uniformly bounded by $K$. Moreover, let $t>0$ be such that 
\[
r(\Sigma)\frac{t}{\lfloor n/2\rfloor} \leq \frac{1}{104}.
\]
Then for any $\sigma\geq \sqrt{K \,\mathrm{r}(\Sigma)}\, \|\Sigma\|$ and 
$\theta:=\theta_\sigma = \frac{1}{\sigma}\sqrt{\frac{2t}{\lfloor n/2\rfloor}}$,
\[
\Big\| \wh \Sigma_\star - \Sigma \Big\|\leq 23\sigma \sqrt{\frac{t}{\lfloor n/2\rfloor}}
\]
with probability $\geq 1 - (4d+1)e^{-t}$.
\end{corollary}
\noindent An adaptive version of the estimator $\widetilde \Sigma_\star$ can be constructed as in \eqref{eq:lepski}, and its performance follows similarly from Theorem \ref{th:lepski}.
\begin{remark}
It is known \cite{koltchinskii2017concentration} that the quantity $\sqrt{\mathrm{r}(\Sigma)} \|\Sigma\|$ controls the expected error of the sample covariance estimator in the Gaussian setting. 
On the other hand, fluctuations of the error around its expected value in the Gaussian case \cite{koltchinskii2017concentration} are controlled by the ``weak variance'' $\sup_{v\in\mb R^d:\|v\|_2=1}\mb E^{1/2}\dotp{Z}{v}^4\leq \sqrt{K} \|\Sigma\|$, while in our bounds fluctuations are controlled by the larger quantity $\sigma^2$; this fact leaves room for improvement in our results.
\end{remark}

\subsection{Estimation in Frobenius norm}
\label{sec:frob}

Next, we show that thresholding the singular values of the adaptive estimator $\widetilde \Sigma_\star$ (defined as in \eqref{eq:lepski} for some $\gamma>1$) yields the estimator that achieves optimal performance in Frobenius norm. 
Given $\tau>0$, define
\begin{align}
\label{eq:thresholding}
&
\widetilde \Sigma_\star^\tau = \sum_{j=1}^d \max\l(\lambda_j\l(\widetilde \Sigma_\star\r) -\tau/2, 0\r) v_j(\widetilde \Sigma_\star) v_j(\widetilde \Sigma_\star)^T,
\end{align}
where $\lambda_j(\widetilde \Sigma_\star)$ and $v_j(\widetilde \Sigma_\star)$ are the eigenvalues and the corresponding eigenvectors of $\widetilde \Sigma_\star$. 
\begin{corollary}
\label{cor:frob}
Assume that the kurtosis of linear forms $\langle Y,v \rangle, \ v\in \mb R^d,$ is uniformly bounded by $K$. Moreover, let $t>0$ be such that 
\[
\mathrm{r}(\Sigma)\frac{t+\Xi}{\lfloor n/2\rfloor} \leq \frac{1}{104},
\]
where $\Xi$ was defined in \eqref{eq:xi} with $\sigma_\ast:=\sqrt{K\, \mathrm{r}(\Sigma)} \|\Sigma\|$. Then for any 
\[
\tau \geq \gamma\cdot 138\sqrt{K}\, \|\Sigma\| \,\sqrt{\frac{\mathrm{r}(\Sigma)(t+\Xi)}{\lfloor n/2\rfloor}},
\]
\begin{align}
&
\label{eq:ex70}
\l\| \widetilde \Sigma_\star^\tau - \Sigma \r\|_{\mathrm{F}}^2\leq 
\inf_{S\in \mb R^{d\times d},S=S^T} \l[  \l\| S - \Sigma \r\|_{\mathrm{F}}^2 + \frac{(1+\sqrt{2})^2}{8}\tau^2\mathrm{rank}(S)  \r].
\end{align}
with probability $\geq 1-(4d+1)e^{-t}$.
\end{corollary}
\noindent The proof of this corollary is given in Section \ref{proof:frob}. 

\subsection{Masked covariance estimation}
\label{sec:masked}

Masked covariance estimation framework is based on the assumption that some entries of the covariance matrix $\Sigma$ are ``more important.'' 
This is quantified by a symmetric mask matrix $M\in \mb R^{d\times d}$, whence the goal is to estimate the matrix 
$M\odot \Sigma$ that ``downweights'' the entries of $\Sigma$ that are deemed less important, or incorporates the prior information on $\Sigma$.  
This problem formulation has been introduced in \cite{levina2012partial}, and later studied in a number of papers including \cite{chen2012masked} and \cite{kabanava2017masked}. 

We will be interested in finding an estimator $\wh \Sigma_\star^M$ such that 
$\| \wh \Sigma_\ast^M - M\odot\Sigma\|$ is small with high probability, and specifically in dependence of the estimation error on the mask matrix $M$. 
Consider the following estimator:
\begin{align}
\label{eq:masked}
\wh \Sigma_\star^M =
\argmin_{S\in \mb R^{d\times d}, S=S^T} \Bigg[ \tr \sum_{i\ne j}
\Psi\left(\theta\left( \frac{ M \odot (Y_i - Y_j)(Y_i - Y_j)^T}{2} - S\right)\right) \Bigg],
\end{align}
which is the ``robust'' version of the estimator $M\odot \widetilde \Sigma_n$, where $\widetilde \Sigma_n$ is the sample covariance matrix defined in \eqref{eq:sample-covariance}. 
Next, following \cite{chen2012masked} we introduce additional parameters that appear in the performance bounds for 
$\wh \Sigma_\star^M$. 
Let 
\[
\|M\|_{1\rightarrow 2}:=\max_{j=1,\ldots,d} \sqrt{\sum_{i=1}^d M_{ij}^2}
\]
be the maximum $\|\cdot\|_2$ norm of the columns of $M$. We also define
\[
\nu_4(Y):=\sup_{\|\mathbf{v}\|_2\leq1}\mb E^{1/4}\langle\mathbf{v},Y -\mb E Y\ \rangle^4
\]
and 
\[
\mu_4(Y)= \max_{j=1\ldots d} \mb E^{1/4}( Y^{(j)} - \mb EY^{(j)})^4. 
\]
The following result describes the finite-sample performance guarantees for $\wh \Sigma_\star^M$. 
\begin{corollary}
\label{cor:masked}
Assume that the kurtosis of the coordinates $Y^{(j)}=\dotp{Y}{e_j}$ of $Y$ is uniformly bounded by $K'$.
Moreover, let $t>0$ be such that 
\[
\sqrt{K'} \frac{\tr(\Sigma)}{\nu_4^2(Y)} \frac{ t}{\lfloor n/2 \rfloor}\leq \frac{1}{104}.
\]
Then for any $\Delta\geq \sqrt{2} \|M\|_{1\rightarrow 2} \,\nu_4(Y) \, \mu_4(Y)$ and 
$\theta=\frac{1}{\Delta}\sqrt{\frac{2t}{\lfloor n/2 \rfloor}}$, 
\[
\Big\| \wh \Sigma_\star^M - M\odot \Sigma \Big\|\leq 23\Delta\sqrt{\frac{t}{\lfloor n/2 \rfloor}}
\]
with probability $\geq 1-(4d+1)e^{-t}$.
\end{corollary}
\begin{proof}
Let $X$ and $X'$ be independent and identically distributed random variables. 
Then it is easy to check that 
\begin{align}
\label{eq:c30}
&
\mb E(X-X')^4\leq 8 \mb E(X-\mb EX)^4.
\end{align}
It implies that $\nu^2_4(Y_1 - Y_2)\leq 2\sqrt{2} \nu_4^2(Y)$ and 
$\mu_4(Y_1-Y_2)\leq 2\sqrt{2} \mu_4(Y)$.  

\noindent Next, Lemma 4.1 in \cite{chen2012masked} yields that 
\begin{align}
\label{eq:c11}
&
\Bigg\| \mb E \l( \frac{(Y_1 - Y_2)(Y_1 - Y_2)^T}{2}\odot M\r)^2 \Bigg\| \leq 
2  \|M\|^2_{1\rightarrow 2} \, \mu^2_4(Y) \, \nu_4^2(Y). 
\end{align}
Next, we will find an upper bound for the trace of $\mb E \l( \frac{(Y_1 - Y_2)(Y_1 - Y_2)^T}{2}\odot M\r)^2$. 
It is easy to see that (e.g., see equation (4.1) in \cite{chen2012masked})  
\[
\mb E \l( \frac{(Y_1 - Y_2)(Y_1 - Y_2)^T}{2}\odot M\r)^2 =
\sum_{j=1}^d M^{(j)} \l(M^{(j)}\r)^T \odot \mb E \l(\frac{Y_1^{(j)} - Y_2^{(j)}}{\sqrt{2}}\r)^2 \frac{(Y_1 - Y_2)(Y_1 - Y_2)^T}{2},
\]
where $M^{(j)}$ denotes the $j$-th column of the matrix $M$. 
It follows from \eqref{eq:c30}, H\"{o}lder's inequality and the bounded kurtosis assumption that 
\begin{align*}
\tr\l[ \mb E \l( \frac{(Y_1 - Y_2)(Y_1 - Y_2)^T}{2}\odot M\r)^2 \r] &=
\sum_{i,j=1}^d M_{i,j}^2 \mb E\l[\l(\frac{Y_1^{(i)} - Y_2^{(i)}}{\sqrt{2}}\r)^2\l(\frac{Y_1^{(j)} - Y_2^{(j)}}{\sqrt{2}}\r)^2 \r] \\
&\leq
2\sum_{i,j=1}^d M_{i,j}^2 \mb E^{1/2} \l(Y^{(i)} - \mb E Y^{(i)}\r)^4 \mb E^{1/2} \l(Y^{(j)} - \mb E Y^{(j)}\r)^4 \\
&\leq 
2\sqrt{K'} \mu_4^2(Y) \|M\|^2_{1\rightarrow 2} \, \tr(\Sigma).
\end{align*}
Next, we deduce that for $\Delta^2 \geq 2  \|M\|^2_{1\rightarrow 2} \, \mu^2_4(Y) \, \nu_4^2(Y)$, 
\[
\frac{\tr\l[ \mb E \l( \frac{(Y_1 - Y_2)(Y_1 - Y_2)^T}{2}\odot M\r)^2 \r]}{\Delta^2}
\leq \sqrt{K'} \frac{\tr(\Sigma)}{\nu_4^2(Y)}.
\]
Result now follows from Theorem \ref{thm:new-performance} and Remark \ref{remark:rank}. 
\end{proof}

\begin{remark}
Let
\[
K:=\sup\limits_{v:\|v\|_2=1}\frac{\mb E \langle Y - \mb EY, v \rangle^4}{\l[\mb E \langle Y - \mb EY, v \rangle^2\r]^2}.
\]
Since $\nu_4^2(Y)\leq \sqrt{K}\| \Sigma \|$ by Lemma \ref{lemma:variance} and 
$\mu_4^2\leq \sqrt{K'} \l\|\Sigma \r\|_{\max}$, we can state a slightly modified version of Corollary \ref{cor:masked}. 
Namely, let $t>0$ be such that  
\[
\sqrt{\frac{K'}{K}} \mathrm{r}(\Sigma) \frac{ t}{\lfloor n/2 \rfloor}\leq \frac{1}{104}.
\]
Then for any $\Delta\geq \sqrt{2K} \|M\|_{1\rightarrow 2} \sqrt{\l\|\Sigma \r\|_{\max} \, \|\Sigma\|}$ and 
$\theta=\frac{1}{\Delta}\sqrt{\frac{2t}{\lfloor n/2 \rfloor}}$, 
\[
\Big\| \wh \Sigma_\star^M - M\odot \Sigma \Big\|\leq 23\Delta\sqrt{\frac{t}{\lfloor n/2 \rfloor}}
\]
with probability $\geq 1-(4d+1)e^{-t}$. 
In particular, if $ \|M\|^2_{1\rightarrow 2} \ll \mathrm{r}(\Sigma)\frac{\|\Sigma\|_{\mbox{\quad }}}{\l\|\Sigma \r\|_{\max}}$, then our bounds show that $M\odot \Sigma$ can be estimated at a faster rate than $\Sigma$ itself. 
This conclusion is consistent with results in \cite{chen2012masked} for Gaussian random vectors (e.g., see Theorem 1.1 in that paper); however, we should note that our bounds were obtained under much weaker assumptions.

\end{remark}

\section{Proofs of the mains results}
\label{section:proofs}

In this section, we present the proofs that were omitted from the main exposition. 

\subsection{Technical tools}

We recall several useful facts from probability theory and matrix analysis that our arguments rely on. 
\begin{fact}
\label{fact:01}
Let $f:\mb R\mapsto \mb R$ be a convex function. 
Then $A\mapsto \tr f(A)$ is convex on the set of self-adjoint matrices. 
In particular, for any self-adjoint matrices $A,B$, 
\[
\tr f\l( \frac{A+B}{2} \r)\leq \frac{1}{2}\tr f(A) + \frac{1}{2}\tr f(B).
\] 
\end{fact}
\begin{proof}
This is a consequence of Peierls inequality, see Theorem 2.9 in \cite{carlen2010trace} and the comments following it. 
\end{proof}

\begin{fact}
\label{fact:05}
Let $F:\mb R\mapsto \mb R$ be a continuously differentiable function, and $S\in \mb C^{d\times d}$ be a self-adjoint matrix. Then the gradient of $G(S):=\tr F(S)$ is 
\[
\nabla G(S) = F'(S),
\]  
where $F'$ is the derivative of $F$ and $F'(S): \mb C^{d\times d}\mapsto \mb C^{d\times d}$ is the matrix function in the sense of the definition \ref{matrix-function}. 
\end{fact}
\begin{proof}
See Lemma A.1 in \cite{minsker2016sub}. 
\end{proof}

\begin{fact}
\label{fact:06}
Function $\psi(x)$ defined in \eqref{eq:psi} satisfies
\begin{align}
\label{eq:psi-ineq}
&
-\log(1-x+x^2)\leq \psi(x) \leq \log(1+x+x^2) 
\end{align}
for all $x\in \mb R$. 
Moreover, as a function of $\mb H^d$-valued argument (see definition \ref{matrix-function}), $\psi(\cdot)$ is Lipschitz continuous in the Frobenius and operator norms with Lipschitz constant $1$, meaning that for all $A_1,A_2\in \mb H^d$,
\begin{align*}
&
\l\| \psi(A_1) - \psi(A_2) \r\|_{\mathrm{F}} \leq \l\| A_1 - A_2 \r\|_{\mathrm{F}},
\\
&
\l\| \psi(A_1) - \psi(A_2) \r\| \leq \l\| A_1 - A_2 \r\|.
\end{align*}
\end{fact}
\begin{proof}
To show \eqref{eq:psi-ineq}, it is enough to check that $x-x^2/2\geq -\log(1-x+x^2)$ for $x\in[0,1]$ and that 
$x-x^2/2\leq \log(1+x+x^2), \ x\in[0,1]$. Other inequalities follow after the change of variable $y=-x$. 
To check that $f(x):=x-x^2/2\geq -\log(1-x+x^2):=g(x)$ for $x\in[0,1]$, note that $f(0)=g(0)=0$ and that 
$f'(x)=1-x\geq 1 - \frac{x(1+x)}{1-x+x^2}=g'(x)$ for $x\in[0,1]$. Inequality $x-x^2/2\leq \log(1+x+x^2), \ x\in[0,1]$ can be established similarly.

Note that the function $\psi:\mb R\mapsto \mb R$ is Lipshitz continuous with Lipschitz constant $1$ as a function of real variable. Lemma 5.5 (Chapter 7) in \cite{bhatia2013matrix} immediately implies that it is also Lipshitz continuous in the Frobenius norm, still with Lipschitz constant $1$. 

Lipshitz property of $\psi$ in the operator norm follows from Corollary 1.1.2 in \cite{aleksandrov2016operator} which states that 
if $g\in C^1(\mb R)$ and $g'$ is positive definite, then the Lipschitz constant of $g$ (as a function on $\mb H^d$) is equal to $g'(0)$. It is easy to check that 
\[
\psi'(x)=\begin{cases}
1- |x|,~&|x|\leq1,\\
0,~&\text{otherwise},
\end{cases}
\]
which is the Fourier transform of the positive integrable function $\mathrm{sinc}(y) = \l(\frac{\text{sin}(\pi y)}{\pi y}\r)^2$, hence $\psi'$ is positive definite and the (operator) Lipschitz constant of $\psi$ is equal to $1$.
\end{proof}

\begin{fact}
\label{fact:02}
Let $T_1,\ldots,T_L$ be arbitrary $\mb H^d$-valued random variables, and $p_1,\ldots,p_L$ be non-negative weights such that $\sum_{j=1}^L p_j=1$. 
Moreover, let $T=\sum_{j=1}^L p_j T_j$ be convex combination of $T_1,\ldots, T_L$. Then 
\[
\Pr\l(\lambda_{\mx}(T)\geq t \r) \leq \max_{j=1,\ldots,L} \l[ \inf_{\theta>0} e^{-\theta t}\mb E \tr e^{\theta T_j}\r].
\]
\end{fact}
\begin{proof}
This fact is a corollary of the well-known Hoeffding's inequality (see Section 5 in \cite{hoeffding1963probability}). 
Indeed, for any $\theta>0$,
\begin{align*}
\pr{\lambda_{\mx}\l(\sum_{j=1}^L p_j T_j \r)\geq t} &\leq 
\pr{\exp\l(\theta \lambda_{\mx}\l( \sum_{j=1}^L p_j T_j \r)\r) \geq e^{\theta t}} \\
& \leq 
e^{-\theta t} \mb E \tr \exp\l(\theta\sum_{j=1}^L p_j T_j \r) 
\leq e^{-\theta t}\sum_{j=1}^L p_j \mb E \tr\exp\l( \theta T_j\r),
\end{align*}
where the last inequality follows from Fact \ref{fact:01}.
\end{proof}

\begin{fact}[Chernoff bound]
\label{fact:03}
Let $\xi_1,\ldots,\xi_n$ be a sequence of i.i.d. copies of $\xi$ such that $\pr{\xi=1}=1-\pr{\xi=0}=p\in (0,1)$, and define 
$S_n:=\sum_{j=1}^n \xi_j$.  
Then 
\[
\pr{S_n/n \geq (1+\tau)p}\leq \inf_{\theta>0}\Big[ e^{-\theta np(1+\tau)}\mb E e^{\theta S_n} \Big]\leq
\begin{cases}
e^{-\frac{\tau^2 np}{2+\tau}}, & \tau>1, \\
e^{-\frac{\tau^2 np}{3}}, & 0<\tau\leq 1.
\end{cases}
\]
\end{fact}
\begin{proof}
See Proposition 2.4 in \cite{angluin1979fast}. 
\end{proof}
Let $\pi_n$ be the collection of all permutations $i:\{1,\ldots,n\}\mapsto \{1,\ldots,n\}$. 
For integers $m\leq \lfloor n/2\rfloor$, let $k=\lfloor n/m \rfloor$. 
Given a permutation $(i_1,\ldots,i_n)\in \pi_n$ and a U-statistic $U_n$ defined in \eqref{u-stat}, let 
\begin{align}
\label{eq:w}
&
W_{i_1,\ldots,i_n}:=\frac{1}{k}\l( H\l(X_{i_1},\ldots,X_{i_m} \r) + H\l(X_{i_{m+1}},\ldots,X_{i_{2m}}\r) + \ldots + 
H\l(X_{i_{(k-1)m+1}},\ldots,X_{i_{km}} \r) \r).
\end{align}
\begin{fact}
\label{fact:04}
The following equality holds:
\[
U_n = \frac{1}{n!}\sum_{(i_1,\ldots,i_n) \in \pi_n}  W_{i_1,\ldots,i_n}.
\]
\end{fact}
\begin{proof}
See Section 5 in \cite{hoeffding1963probability}. 
\end{proof}

Let $Z_1,\ldots,Z_n$ be a sequence of independent copies of $Z\in \mb H^d$ such that $\l\| \mb EZ^2\r\|<\infty$. 
\begin{fact}[Matrix Bernstein Inequality]
\label{fact:bernstein}
Assume that $\|Z -\mb EZ\|\leq M$ almost surely. Then for any $\sigma\geq \l\| \mb E(Z - \mb EZ)^2 \r\|$,
\[
\Bigg\| \frac{\sum_{j=1}^n Z_j}{n} - \mb EZ \Bigg\|\leq 2\sigma\sqrt{\frac{t}{n}}\bigvee \frac{4Mt}{3n}
\]
with probability $\geq 1 - 2de^{-t}$.
\end{fact}
\begin{proof}
See Theorem 1.4 in \cite{tropp2012user}.
\end{proof}
Assume that $\l\| H\l(X_{i_1},\ldots, X_{i_m}\r) \r\|\leq M$ almost surely. Together with Facts \ref{fact:04} and \ref{fact:02}, Bernstein's inequality can be used to show that 
\begin{align}
\label{eq:p10}
&
\Big\| U_n - \mb EH \Big\|\leq 2 \big\| \mb E(H - \mb E H)^2 \big\|^{1/2}\sqrt{\frac{t}{k}} \bigvee \frac{4Mt}{3n}
\end{align}
with probability $\geq 1 - 2de^{-t}$. This corollary will be useful in the sequel.
\begin{fact}
\label{fact:legacy}
Let $\psi(\cdot)$ be defined by \eqref{eq:psi}. Then the following inequalities hold for all $\theta>0$:
\begin{align*}
&
\mb E \tr \exp\l( \sum_{j=1}^n \l( \psi(\theta Z_j) -\theta\mb EZ\r) \r)\leq \tr \exp\l( n\theta^2 \mb EZ^2 \r), \\
&
\mb E \tr \exp\l( \sum_{j=1}^n \l( \theta\mb EZ - \psi(\theta Z_j)\r) \r)\leq \tr \exp\l( n \theta^2 \mb EZ^2 \r).
\end{align*}
\end{fact}
\begin{proof}
These inequalities follow from \eqref{eq:psi-ineq} and Lemma 3.1 in \cite{minsker2016sub}. 
Note that we did not assume boundedness of $\|Z -\mb EZ\|\leq M$ above.
\end{proof}
Finally, we will need the following statement related to the self-adjoint dilation \eqref{eq:dilation}.
\begin{fact}
\label{fact:dilation}
Let $S\in \mb C^{d_1\times d_1}, \ T\in \mb C^{d_2\times d_2}$ be self-adjoint matrices, and $A\in \mb C^{d_1\times d_2}$. 
Then 
\[
\l\| \begin{pmatrix}
S & A \\
A^\ast & T
\end{pmatrix}\r\|
\geq
\l\| \begin{pmatrix}
0 & A \\
A^\ast & 0
\end{pmatrix} \r\|.
\]
\end{fact}
\begin{proof}
See Lemma 2.1 in \cite{minsker2016sub}.
\end{proof}

\subsection{Proof of Lemma \ref{lemma:optimization}}
\label{proof:opt}

(1) Convexity follows from Fact \ref{fact:01} since the sum of convex functions is a convex function. \\
(2) The expression for the gradient follows from Fact \ref{fact:05}. 
To show that $\nabla F_\theta(U)$ is Lipschitz continuous, note that 
\begin{multline*}
\Big\| \frac{1}{\theta}\psi\l( \theta\l(H_{i_1,\ldots,i_m} - U_1 \r)\r)  - \frac{1}{\theta}\psi\l( \theta\l(H_{i_1,\ldots,i_m} - U_2 \r)\r) \Big\| 
\\
\leq \Big\| \frac{1}{\theta} \l( \theta\l(H_{i_1,\ldots,i_m} - U_1 \r) - \theta\l(H_{i_1,\ldots,i_m} - U_2 \r) \r) \Big\|  
= \l\| U_1 - U_2 \r\|, 
\end{multline*}
\begin{multline*}
\Big\| \frac{1}{\theta}\psi\l( \theta\l(H_{i_1,\ldots,i_m} - U_1 \r)\r)  - \frac{1}{\theta}\psi\l( \theta\l(H_{i_1,\ldots,i_m} - U_2 \r)\r) \Big\|_{\mathrm{F}} 
\\
\leq 
\Big\| \frac{1}{\theta} \l( \theta\l(H_{i_1,\ldots,i_m} - U_1 \r) - \theta\l(H_{i_1,\ldots,i_m} - U_2 \r) \r) \Big\|_{\mathrm{F}} 
= \l\| U_1 - U_2 \r\|_{\mathrm{F}}
\end{multline*} 
by Fact \ref{fact:06}. Since the convex combination of Lipschitz continuous functions is still Lipschitz continuous, the claim follows. \\
(3) Since $\wh U_n^\star$ is the solution of the problem \eqref{eq:estimator1}, the directional derivative 
\[
dF_\theta(\wh U_n^\star; B) := \lim_{t\to 0}\frac{F_\theta(\wh U_n^\star + tB) - F_\theta(\wh U_n^\star)}{t} 
= \tr \l( \nabla F_\theta(\wh U_n^\star) \, B\r)
\]
is equal to 0 for any $B\in \mb H^d$. Result follows by taking consecutively $B_{i,j}=e_i e_j^T + e_j e_i^T, \ i\ne j$ and $B_{i,i} = e_i e_i^T, \ i=1,\ldots,d$, where $\l\{ e_1,\ldots,e_d\r\}$ is the standard Euclidean basis. 
\qed

\subsection{Proof of Theorem \ref{thm:new-performance}}
\label{proof:new-performance}

The proof is based on the analysis of the gradient descent iteration for the problem \eqref{eq:estimator1}. 
Let 
\[
G(U) := \tr F_\theta(U) = \tr \l[ \frac{1}{\theta^2}\frac{(n-m)!}{n!}\sum_{(i_1,\ldots,i_m)\in I^m_n}\Psi\Big( \theta \left(H(X_{i_1},\ldots,X_{i_m}) - U \right) \Big) \r],
\]
and define 
\begin{align*}
U^{(0)}_n :&= \mb EH = \mb E H(X_1,\ldots,X_m), \\
U^{(j)}_n :&= U^{(j-1)}_n - \nabla G\l( U^{(j-1)}_n \r) \\
&
=U^{(j-1)}_n+\frac{1}{\theta}\frac{(n-m)!}{n!}\sum_{(i_1,\ldots,i_m)\in I^m_n}\psi\Big(\theta\l( H_{i_1\ldots i_m} - U^{(j-1)}_n\r) \Big), \ j\geq 1,
\end{align*}
which is the gradient descent for \eqref{eq:estimator1} with the step size equal to $1$. 
We will show that with high probability (and for an appropriate choice of $\theta$), $U^{(j)}_n$ does not escapes a small neighborhood of $\mb E H(X_1,\ldots,X_m)$. 
The claim of the theorem then easily follows from this fact. 

\noindent Give a permutation $(i_1,\ldots,i_n)\in \pi_n$ and $U\in\mb H^d$, let $k = \lfloor n/m \rfloor$ and
\begin{align*}
Y_{i_1\ldots i_m}(U;\theta)&:=\psi\l(\theta\l(H_{i_1\ldots i_m}-U\r)\r), 
\\
W_{i_1\ldots i_n}(U;\theta)& := \frac{1}{k}\l( Y_{i_1\ldots i_m}(U;\theta) + Y_{i_{m+1}\ldots i_{2m}}(U;\theta) + \ldots + Y_{i_{(k-1)m+1}\ldots i_{km}}(U;\theta) \r).
\end{align*}
Fact \ref{fact:04} implies that
\begin{equation}
\label{inter-permutation}
\nabla G\l( U \r)=\frac{(n-m)!}{n!}\sum_{(i_1\ldots i_m)\in I^m_n} \frac{1}{\theta}\psi\Big(\theta\l( H_{i_1\ldots i_m} - U\r) \Big) = \frac{1}{n!}\sum_{(i_1\ldots i_n)\in\pi_n}\frac{1}{\theta} W_{i_1 \ldots i_n}(U;\theta),
\end{equation}
where $\pi_n$ ranges over all permutations of $(1,\ldots,n)$. 
Next, for $j\geq 1$ we have 
\begin{align}
\label{eq:b10}
\nonumber
\Big\| U_n^{(j)} - \mb EH \Big\| & = 
\Bigg\|  \frac{1}{\theta}\frac{(n-m)!}{n!}\sum_{(i_1,\ldots,i_m)\in I^m_n}\psi\Big(\theta\l( H_{i_1\ldots i_m} - U^{(j-1)}_n\r) 
- \l( \mb EH - U^{(j-1)}_n \r) \Bigg\| \\
&
=\Bigg\| \frac{1}{\theta n!}\sum_{(i_1\ldots i_n)\in\pi_n}W_{i_1 \ldots i_n}(U_n^{(j-1)};\theta) -  \l( \mb EH - U^{(j-1)}_n \r) \Bigg\| \\
& \nonumber
\leq 
\Bigg\| \frac{1}{\theta n!}\sum_{(i_1\ldots i_n)\in\pi_n} 
\Big( W_{i_1 \ldots i_n}(U_n^{(j-1)};\theta) - W_{i_1,\ldots,i_n}(\mb EH;\theta_\sigma) \Big) -  \l( \mb EH - U^{(j-1)}_n \r) \Bigg\| \\
& \nonumber
+ \Bigg\| \frac{1}{\theta_\sigma}\frac{1}{n!}\sum_{\pi_n} W_{i_1,\ldots,i_n}(\mb E H;\theta_\sigma) \Big)  \Bigg\|.
\end{align}
The following two lemmas provide the bounds that allows to control the size of $\Big\| U_n^{(j)} - \mb EH \Big\|$. 
For a given $\sigma^2\geq \l\| \mb E(H-\mb EH)^2 \r\|$ and $\theta_\sigma = \frac{1}{\sigma}\sqrt{\frac{2 t}{k}}$, consider the random variable
\[
L_n(\delta) = \sup_{\|U-\mb EH\|\leq\delta} \l\| \frac{1}{\theta_\sigma}\frac{1}{n!}\sum_{\pi_n}
\Big( W_{i_1,\ldots,i_n}(U;\theta_\sigma) 
-W_{i_1,\ldots,i_n}(\mb EH;\theta_\sigma) \Big)   - (\mb EH - U) 
\r\|.
\]
\begin{lemma}
\label{lemma:sup-bound}
With probability $\geq 1 - (2d+1)e^{-t} $, for all $\delta\leq \frac{1}{2}\frac{1}{\theta_\sigma}$ simultaneously, 
\[
L_n(\delta) \leq \l( \mathrm{r}_H\frac{26 t}{k}+ \frac{1}{2}\r)\delta 
+ \frac{3(1+\sqrt{2})}{2} \sigma \sqrt{\frac{t}{k}}.
\]
\end{lemma}
\noindent The proof of this lemma is given in Section \ref{proof:lemma-sup-bound}. 
\begin{lemma}
\label{lemma:EH}
With probability $\geq 1 - 2d e^{-t}$, 
\[
\Bigg\| \frac{1}{\theta_\sigma}\frac{1}{n!}\sum_{\pi_n} W_{i_1,\ldots,i_n}(\mb E H;\theta_\sigma) \Big)  \Bigg\| \leq 
\frac{3}{\sqrt{2}}\sigma\sqrt{\frac{t}{k}}.
\]
\end{lemma}
\noindent The proof is given in Section \ref{proof:lemma-EH}. 
Next, define the sequence 
\begin{align*}
\delta_0 & = 0, \\
\delta_j & = \l( \mathrm{r}_H\frac{26 t}{k} + \frac{1}{2}\r)\delta_{j-1} + 5.75 \sigma\sqrt{\frac{t}{k}}.
\end{align*}
If $\mathrm{r}_H\frac{26 t}{k}\leq \frac{1}{4}$, then $t\leq \frac{k}{104}$, hence 
$5.75 \sigma\sqrt{\frac{t}{k}} \leq \frac{1}{8} \frac{1}{\theta_\sigma}$ and 
\[
\delta_j \leq \frac{3}{4}\delta_{j-1} + \frac{1}{8} \frac{1}{\theta_\sigma} \leq \frac{1}{2} \frac{1}{\theta_\sigma} 
\]
for all $j\geq 0$.
Let $\m E_0$ be the event of probability $\geq 1 - (4d+1)e^{-t}$ on which the inequalities of Lemmas \ref{lemma:sup-bound} and \ref{lemma:EH} hold. 
It follows from \eqref{eq:b10}, Lemma \ref{lemma:sup-bound} and Lemma \ref{lemma:EH} that on the event $\m E_0$, for all $j\geq 1$
\begin{align*}
\Big\| U_n^{(j)} - \mb EH \Big\| &\leq L_n\l( \l\| U_n^{(j-1)} - \mb EH \r\|\r) 
+\Bigg\| \frac{1}{\theta_\sigma}\frac{1}{n!}\sum_{\pi_n} W_{i_1,\ldots,i_n}(\mb E H;\theta_\sigma) \Big)  \Bigg\| \\
&\leq  
\l( \mathrm{r}_H\frac{26 t}{k}+ \frac{1}{2}\r)\delta_{j-1} + \frac{3(1+2\sqrt{2})}{2} \sigma \sqrt{\frac{t}{k}}
\leq \delta_j
\end{align*}
given that $\mathrm{r}_H\frac{26 t}{k} \leq \frac{1}{4}$; we have also used the numerical bound $\frac{3(1+2\sqrt{2})}{2}\leq 5.75$. 

\noindent Finally, it is easy to see that for all $j\geq 1$ and 
$\gamma = \mathrm{r}_H\frac{26 t}{k}+\frac{1}{2}\leq \frac{3}{4}$,
\begin{align}
\label{eq:b30}
&
\delta_j  = \delta_0 \gamma^j + \sum_{l=0}^{j-1} \gamma^l \cdot 5.75\sigma\sqrt{\frac{t}{k}}
\leq \sum_{l\geq 0} (3/4)^{l} \cdot 5.75\sigma\sqrt{\frac{t}{k}} \leq 23\sigma\sqrt{\frac{t}{k}}.
\end{align}
Since $U_n^{(j)}\to \wh U_n^\star$ pointwise as $j\to\infty$, the result follows. 

\subsection{Proof of Lemma \ref{lemma:sup-bound}}
\label{proof:lemma-sup-bound}

Recall that $\sigma^2\geq \l\| \mb E(H_{i_1,\ldots,i_m} -\mb EH)^2 \r\|$, $\theta_\sigma:=\frac{1}{\sigma}\sqrt{\frac{2 t}{k}}$, and
\[
\psi(\theta_\sigma x)=\begin{cases}
\theta_\sigma x - \sign(x)\frac{\theta_\sigma^2 x^2}{2}, & x\in[-1/\theta_\sigma,1/\theta_\sigma], \\
1/2, & |x|>1/\theta_\sigma. 
\end{cases}
\]
The idea of the proof is to exploit the fact that $\psi(\theta_\sigma x)$ is ``almost linear'' whenever $x\in[-1/\theta_\sigma,1/\theta_\sigma]$, and its nonlinear part is active only for a small number of multi-indices $(i_1,\ldots,i_m)\in I_n^m$. 
Let 
\[
\chi_{i_1,\ldots,i_m} = I\l\{ \l\| H_{i_1,\ldots,i_m} - \mb EH \r\| \leq \frac{1}{2\theta_\sigma} \r\}.
\] 
Note that by Chebyshev's inequality, and taking into account the fact that 
\[
\| H_{i_1,\ldots,i_m} - \mb EH\|\leq \| H_{i_1,\ldots,i_m}-\mb EH \|_{\mathrm{F}},
\]
\begin{align}
\label{eq:a30}
\nonumber
\pr{\chi_{i_1,\ldots,i_m} = 0 }&\leq 4\theta_\sigma^2\mb E \l\| H_{i_1,\ldots,i_m} - \mb EH \r\|^2_\F  \\
&
\leq \frac{8t}{k}\frac{\tr \l( \mb E(H_{i_1,\ldots,i_m}-\mb EH)^2 \r)}{\l\| \mb E(H_{i_1,\ldots,i_m}-\mb EH)^2 \r\|} 
= \mathrm{r}_H \frac{8t}{k}.
\end{align}
Define the event 
\[
\m E=\l\{ \sum_{(i_1,\ldots,i_m)\in I_n^m} \Big( 1 - \chi_{i_1,\ldots,i_m} \Big) \leq \mathrm{r}_H\frac{8t }{k} \frac{n!}{(n-m)!}\cdot \l(1 + \sqrt{\frac{3}{8 \mathrm{r}_H}}\r) \r\}.
\] 
We will apply a version of Chernoff bound to the $\mb R$-valued U-statistic 
$\frac{(n-m)!}{n!}\sum_{(i_1,\ldots,i_m)\in I_n^m} \l( 1 - \chi_{i_1,\ldots,i_m} \r)$. 
A combination of Fact \ref{fact:04}, Fact \ref{fact:02} applied in the scalar case $d=1$, and Fact \ref{fact:03} implies that 
\[
\Pr\l( \frac{(n-m)!}{n!}\sum_{(i_1,\ldots,i_m)\in I_n^m} \l( 1 - \chi_{i_1,\ldots,i_m} \r) \geq \mathrm{r}_H\frac{8t }{k} \cdot \l(1 + \tau\r) \r) 
\leq e^{-\tau^2 8t \, \mathrm{r}_H/3 } 
\]
for $0<\tau<1$. Hence, choosing $\tau = \sqrt{ \frac{3}{8 \mathrm{r}_H}}$ implies that $\pr{\m E} \geq1- e^{-t}$. 

By triangle inequality, whenever $\chi_{i_1,\ldots,i_m} =1$ and 
$\delta\leq \frac{1}{2}\frac{1}{\theta_\sigma}$, it holds that $\l\|  H_{i_1,\ldots,i_m} - U \r\| \leq \frac{1}{\theta_\sigma}$ for any $U$ such that  $\| U - \mb EH \|\leq \delta$, 
and consequently
\[
\frac{1}{\theta_\sigma}\psi(\theta_\sigma (H_{i_1,\ldots,i_m} - U)) = (H_{i_1,\ldots,i_m} - U) - 
\frac{\theta_\sigma}{2}\sign \l( H_{i_1,\ldots,i_m} - U \r) \l( H_{i_1,\ldots,i_m} - U \r)^2.
\]
Denoting 
\[
S_{i_1,\ldots,i_m}(U) := \sign \l( H_{i_1,\ldots,i_m} - U \r) \l( H_{i_1,\ldots,i_m} - U \r)^2
\] 
for brevity, we deduce that 
\begin{multline*}
\frac{1}{\theta_\sigma}\frac{1}{n!}\sum_{\pi_n}\Big( W_{i_1,\ldots,i_n}(U;\theta_\sigma) 
-W_{i_1,\ldots,i_n}(\expect{H};\theta_\sigma) \Big)   - (\mb EH - U) 
 \\
= \frac{(n-m)!}{n!}\sum_{(i_1,\ldots,i_m)\in I_n^m} \l( \frac{\theta_\sigma}{2}S_{i_1,\ldots,i_m}(\mb EH) - \frac{\theta_\sigma}{2}S_{i_1,\ldots,i_m}(U) \r)\chi_{i_1,\ldots,i_m} \\
+\frac{1}{\theta_\sigma}\frac{(n-m)!}{n!} \sum_{(i_1,\ldots,i_m)\in I_n^m} \l( 1 - \chi_{i_1,\ldots,i_m} \r) 
\Big( Y_{i_1,\ldots,i_m}(U;\theta_\sigma) 
-Y_{i_1,\ldots,i_m}(\expect{H};\theta_\sigma) \Big)   \\
-\frac{(n-m)!}{n!} \sum_{(i_1,\ldots,i_m)\in I_n^m} \l( 1 - \chi_{i_1,\ldots,i_m} \r) \l( \mb EH - U \r).
\end{multline*}
We will separately control the terms on the right hand side of the equality above. 
First, note that on event $\m E$,
\begin{align}
\label{eq:a40}
&
\l\| \frac{(n-m)!}{n!} \sum_{(i_1,\ldots,i_m)\in I_n^m} \l( 1 - \chi_{i_1,\ldots,i_m} \r) \l( \mb EH - U \r)\r\| \leq 
 \mathrm{r}_H\frac{8t }{k}\cdot \l(1 + \sqrt{\frac{3}{8 \mathrm{r}_H }}\r)\delta \leq \mathrm{r}_H \frac{13 t}{k}\delta
\end{align}
since $\| \mb EH - U\|\leq \delta$. 
Next, recalling that $\psi(\cdot)$ is operator Lipschitz (by Fact \ref{fact:06}), wee see that for any $(i_1,\ldots,i_m)\in I_n^m$ 
\[
\frac{1}{\theta_\sigma}\Big\| Y_{i_1,\ldots,i_m}(U;\theta_\sigma) -Y_{i_1,\ldots,i_m}(\mb E{H};\theta_\sigma) \Big\| \leq 
\l\| \mb EH - U\r\| \leq \delta,  
\]
hence on event $\m E$, 
\begin{multline}
\label{eq:a50}
\frac{1}{\theta_\sigma}\frac{(n-m)!}{n!} \l\| \sum_{(i_1,\ldots,i_m)\in I_n^m} \l( 1 - \chi_{i_1,\ldots,i_m} \r) 
\Big( Y_{i_1,\ldots,i_m}(U;\theta_\sigma) 
-Y_{i_1,\ldots,i_m}(\mb E{H};\theta_\sigma) \Big) \r\|  \\
\leq 
\mathrm{r}_H \frac{8t }{k}\cdot \l(1 + \sqrt{\frac{3}{8 \mathrm{r}_H}}\r)\delta \leq \mathrm{r}_H\frac{13 t }{k}\delta.
\end{multline}
Finally, it remains to control the term 
\begin{align*}
&
\m Q(\delta):=\sup_{\|U-\mathbb{E}H\|\leq\delta} \l\| \frac{(n-m)!}{n!}\sum_{(i_1,\ldots,i_m)\in I_n^m} \l( \frac{\theta_\sigma}{2}S_{i_1,\ldots,i_m}(\mb EH) - \frac{\theta_\sigma}{2}S_{i_1,\ldots,i_m}(U) \r)\chi_{i_1,\ldots,i_m} \r\|.  
\end{align*}

\begin{lemma}
\label{lemma:supplement}
With probability $\geq 1 - 2de^{-t}$,
\[
\m Q(\delta) \leq \frac{3(1+\sqrt{2})}{2} \sigma \sqrt{\frac{t}{k}} + \frac{\delta}{2}.
\]
\end{lemma}
\begin{proof}
Observe that for all $U\in \mb H^d$ and $(i_1,\ldots,i_m)\in I_n^m$,
\begin{align*}
& 
- \l( H_{i_1,\ldots,i_m} - U \r)^2\preceq \sign \l( H_{i_1,\ldots,i_m} - U \r) \l( H_{i_1,\ldots,i_m} - U \r)^2 
\preceq \l( H_{i_1,\ldots,i_m} - U \r)^2,
\end{align*}
hence
\begin{multline*}
\Bigg\|\frac{(n-m)!}{n!}\sum_{(i_1,\ldots,i_m)\in I_n^m} \l( \frac{\theta_\sigma}{2}S_{i_1,\ldots,i_m}(\mb EH) - \frac{\theta_\sigma}{2}S_{i_1,\ldots,i_m}(U) \r)\chi_{i_1,\ldots,i_m} \Bigg\| \\
\leq 
\frac{(n-m)!}{n!}  \Bigg\| \sum_{(i_1,\ldots,i_m)\in I_n^m}   \frac{\theta_\sigma}{2} \l( H_{i_1,\ldots,i_m} - U \r)^2 \chi_{i_1,\ldots,i_m} \Bigg\|  \\
+ \frac{(n-m)!}{n!}  \Bigg\| \sum_{(i_1,\ldots,i_m)\in I_n^m}   \frac{\theta_\sigma}{2} \l( H_{i_1,\ldots,i_m} - \mb EH \r)^2 \chi_{i_1,\ldots,i_m} \Bigg\| .
\end{multline*}
Moreover, 
\begin{align*}
&
 \l( H_{i_1,\ldots,i_m} - U \r)^2 \preceq 2\l( H_{i_1,\ldots,i_m} - \mb EH \r)^2 + 2\l( U - \mb EH\r)^2, 
\end{align*}
implying that 
\begin{multline*}
\frac{(n-m)!}{n!}  \Bigg\| \sum_{(i_1,\ldots,i_m)\in I_n^m}   \frac{\theta_\sigma}{2} \l( H_{i_1,\ldots,i_m} - U \r)^2 \chi_{i_1,\ldots,i_m} \Bigg\| \\
\leq 
2\frac{(n-m)!}{n!}  \Bigg\| \sum_{(i_1,\ldots,i_m)\in I_n^m}   \frac{\theta_\sigma}{2} \l( H_{i_1,\ldots,i_m} - \mb EH \r)^2 \chi_{i_1,\ldots,i_m} \Bigg\| + \theta_\sigma\Big\| U - \mb EH \Big\|^2.
\end{multline*}
Hence, we have shown that 
\begin{align}
\label{eq:a51}
&
\m Q(\delta) \leq 3\frac{(n-m)!}{n!}  \Bigg\| \sum_{(i_1,\ldots,i_m)\in I_n^m}   \frac{\theta_\sigma}{2} \l( H_{i_1,\ldots,i_m} - \mb EH \r)^2 \chi_{i_1,\ldots,i_m} \Bigg\| + \theta_\sigma \delta^2.
\end{align}
Since $\delta\leq \frac{1}{2\theta_\sigma}$,
\begin{align}
\label{eq:a52}
&
\theta_\sigma \delta^2 \leq \frac{\delta}{2}.
\end{align}
Next, we will estimate the first term in \eqref{eq:a51} as follows: 
\begin{multline*}
3\frac{(n-m)!}{n!}  \Bigg\| \sum_{(i_1,\ldots,i_m)\in I_n^m}   \frac{\theta_\sigma}{2} \l( H_{i_1,\ldots,i_m} - \mb EH \r)^2 \chi_{i_1,\ldots,i_m} \Bigg\| \\
\leq 
3\frac{(n-m)!}{n!}  \Bigg\| \sum_{(i_1,\ldots,i_m)\in I_n^m}   
\frac{\theta_\sigma}{2} \bigg[ \l( H_{i_1,\ldots,i_m} - \mb EH \r)^2 \chi_{i_1,\ldots,i_m}  - 
\mb E \l[ \l( H_{i_1,\ldots,i_m} - \mb E H \r)^2 
\chi_{i_1,\ldots,i_m} \r] \bigg] \Bigg\| \\
+\frac{3\theta_\sigma}{2} \Big\| \mb E\l[ \l( H_{i_1,\ldots,i_m} - \mb EH \r)^2 \chi_{i_1,\ldots,i_m}\r]  \Big\|.
\end{multline*}
Clearly, $\Big\| \mb E \l[ \l( H_{i_1,\ldots,i_m} - \mb EH \r)^2\chi_{i_1,\ldots,i_m} \r] \Big\|\leq \sigma^2$, hence
\begin{align}
\label{eq:a53}
\frac{3\theta_\sigma}{2} \Big\| \mb E \l[ \l( H_{i_1,\ldots,i_m} - \mb EH \r)^2\chi_{i_1,\ldots,i_m} \r] \Big\| \leq \frac{3\sigma}{2}\sqrt{\frac{2t}{k}}.
\end{align}
The remaining part will be estimated using the Matrix Bernstein's inequality (Fact \ref{fact:bernstein}).
	
\noindent To this end, note that by the definition of $\chi_{i_1,\ldots,i_m}$,
\[
\Big\| \l( H_{i_1,\ldots,i_m} - \mb EH \r)^2 \chi_{i_1,\ldots,i_m}  - \mb E \l[ \l( H_{i_1,\ldots,i_m} - \mb EH \r)^2 \chi_{i_1,\ldots,i_m} \r]\Big\|
\leq 
\l( \frac{1}{2\theta_\sigma}\r)^2
\]
almost surely. 
Moreover, 
\begin{multline*}
\Big\| \mb E\l( \l( H_{i_1,\ldots,i_m} - \mb EH \r)^2 \chi_{i_1,\ldots,i_m}  - \mb E \l[ \l( H_{i_1,\ldots,i_m} - \mb EH \r)^2 \chi_{i_1,\ldots,i_m} \r] \r)^2\Big\| \\
\leq 
\Big\| \mb E\l( \l( H_{i_1,\ldots,i_m} - \mb EH \r)^2 \chi_{i_1,\ldots,i_m} \r)^2\Big\|
\leq \l( \frac{1}{2\theta_\sigma}\r)^2 \, \big\| \mb E\l( H_{i_1,\ldots,i_m} - \mb EH \r)^2\big\|,
\end{multline*}
where we used the fact that 
\begin{align*}
&
\l( \l( H_{i_1,\ldots,i_m} - \mb EH \r)^2 \chi_{i_1,\ldots,i_m} \r)^2 \preceq 
\l( \frac{1}{2\theta_\sigma}\r)^2  \l( H_{i_1,\ldots,i_m} - \mb EH \r)^2.
\end{align*}
Applying the Matrix Bernstein inequality (Fact \ref{fact:bernstein}), we get that with probability $\geq 1 - 2d e^{-t}$
\begin{multline}
\label{eq:a60}
3\frac{(n-m)!}{n!}  \Bigg\| \sum_{(i_1,\ldots,i_m)\in I_n^m}   
\frac{\theta_\sigma}{2} \l[ \l( H_{i_1,\ldots,i_m} - \mb EH \r)^2 \chi_{i_1,\ldots,i_m}  - \mb E \l( H_{i_1,\ldots,i_m} - \mb EH \r)^2 
\chi_{i_1,\ldots,i_m} \r] \Bigg\| \\
\leq 
\frac{3\theta_\sigma}{2}\l[ \frac{2}{2\theta_\sigma} \l\| \mb E(H_{i_1,\ldots,i_m} - \mb EH)^2 \r\|^{1/2} \sqrt{\frac{t}{k}} \bigvee \frac{4}{3}\frac{t}{k}\frac{1}{(2\theta_\sigma)^2}\r] 
\leq \frac{3}{2}\sigma \sqrt{\frac{t}{k}}.
\end{multline}
The bound of Lemma \ref{lemma:supplement} now follows from the combination of bounds \eqref{eq:a52}, \eqref{eq:a53}, \eqref{eq:a60} and \eqref{eq:a51}.
\qed

\noindent Combining the bound of Lemma \ref{lemma:supplement} with \eqref{eq:a40} and \eqref{eq:a50}, 
we get the desired result of Lemma \ref{lemma:sup-bound}.
	
\end{proof}

\subsection{Proof of Lemma \ref{lemma:EH}}
\label{proof:lemma-EH}

Fact \ref{fact:02} implies that for all $s>0$,
\begin{align}
\label{eq:e10}
\Pr\l( \lambda_{\mx}\l( \frac{1}{\theta_\sigma}\frac{1}{n!}\sum_{\pi_n} W_{i_1,\ldots,i_n}(\mb E H;\theta_\sigma) \Big)  \r)\geq 
s \r) 
& \nonumber
\leq 
\inf_{\theta>0} \l[ e^{-\theta s} \mb E\tr e^{(\theta/\theta_\sigma)\,W_{1,\ldots,n}(\mb EH,\theta_\sigma)}\r] \\
& \leq e^{-\theta_\sigma s \, k} \,\mb E\tr e^{k\,W_{1,\ldots,n}(\mb EH,\theta_\sigma)}.
\end{align}
Since 
\[
W_{1,\ldots,n}(\mb EH,\theta_\sigma) = 
\frac{1}{k}\l( \psi\l( \theta_\sigma( H_{1,\ldots,m} -\mb EH)\r) +  \ldots + \psi\l( \theta_\sigma( H_{(k-1)m+1,\ldots,km} -\mb EH)\r) \r)
\]
is a sum of $k$ independent random matrices, we can apply the first inequality of Fact \ref{fact:legacy} to deduce that 
\[
\mb E\tr e^{k\,W_{1,\ldots,n}(\mb EH,\theta_\sigma)} \leq
\tr \exp\l( k\theta_\sigma^2 \mb E(H - \mb EH)^2 \r) \leq d \exp\l( k\theta_\sigma^2 \sigma^2\r),
\]
where we used the fact that $\tr(A)\leq d \|A\|$ for $\mb H^{d\times d} \ni A\succeq 0$. 
Finally, setting $s=\frac{3}{\sqrt{2}}\sigma\sqrt{\frac{t}{k}}$, we obtain from \eqref{eq:e10} that
\[
\Pr\l( \lambda_{\mx}\l( \frac{1}{\theta_\sigma}\frac{1}{n!}\sum_{\pi_n} W_{i_1,\ldots,i_n}(\mb E H;\theta_\sigma) \Big)  \r)\geq 
s \r) 
\leq d e^{-t}.
\]
Similarly, since $-\lambda_{\mn}(A)=\lambda_{\mx}(-A)$ for $A\in \mb H^{d\times d}$, it follows from the second inequality of Fact \ref{fact:legacy} that 
\begin{align*}
\Pr\Bigg( \lambda_{\mn} & \l(\frac{1}{\theta_\sigma}\frac{1}{n!}\sum_{\pi_n} W_{i_1,\ldots,i_n}(\mb E H;\theta_\sigma) \Big)  \r) \leq -s\Bigg) \\
& = \Pr\l( \lambda_{\mx}\l( -\frac{1}{\theta_\sigma}\frac{1}{n!}\sum_{\pi_n} W_{i_1,\ldots,i_n}(\mb E H;\theta_\sigma) \Big)  \r) \geq s\r) \\
&
\leq e^{-\theta_\sigma s\,k}\,\mb E \tr \exp\l(k W_{1,\ldots,n}(\mb EH,\theta_\sigma) \r)\\
&
\leq d e^{-\theta_\sigma s\,k}\, \exp\l( k\theta_\sigma^2 \sigma^2\r)
\leq d e^{-t}
\end{align*}
for $s=\frac{3}{\sqrt{2}}\sigma\sqrt{\frac{t}{k}}$, and result follows.

\subsection{Proof of Lemma \ref{lemma:grad-descent}}
\label{proof:grad-descent}

Part (a) follows from a well-known result (e.g., \cite{bertsekas2009convex}) which states that, given a convex, differentiable function $G: \mb R^{D}\to \mb R$ such that its gradient satisfies 
$\Big\| \nabla G(U_1) - \nabla G(U_2) \Big\|_2 \leq L \| U_1 - U_2 \|_2$, the $j$-th iteration $U^{(j)}$ of the gradient descent algorithm run 
with step size $\alpha\leq \frac{1}{L}$ satisfies 
\[
G\l( U^{(j)} \r) - G(U_\ast) \leq \frac{\l\| U^{(0)} - U_\ast\r\|_2^2}{2\alpha j},
\]
where $U_\ast = \argmin \, G(U)$. 
\\
The proof of part (b) follows the lines of the proof of Theorem \ref{thm:new-performance}: more specifically, the claim follows from equation \eqref{eq:b30}.
\qed

\subsection{Proof of Corollary \ref{cor:rectangular}}
\label{proof:rectangular}

\begin{proof}
Note that 
\[
\big\| \mb E\,\m D(H_{i_1\ldots i_m})^2 \big\| = \max\l( \big\|  \mb E H_{i_1\ldots i_m} \,H_{i_1\ldots i_m}^\ast \big\|, \big\|  \mb E H_{i_1\ldots i_m}^\ast \, H_{i_1\ldots i_m} \big\| \r).
\] 
We apply Theorem \ref{thm:new-performance} applied to self-adjoint random matrices 
\[
\m D(H_{i_1\ldots i_m})\in \mb C^{(d_1+d_2)\times (d_1+d_2)}, \ (i_1,\ldots,i_m)\in I_n^m,
\] 
and obtain that 
\[
\big\| \bar U_n^\star - \m D(\mb E H) \big\| \leq 15\sigma \sqrt{\frac{t}{k}}
\]
with probability $\geq 1 - \l( 2(d_1+d_2)+1 \r)e^{-t}$. 
It remains to apply Fact \ref{fact:dilation}: 
\begin{align*}
\l\| \bar U_n^\star - \m D(\mb E H)\r\| & =
\l\| \begin{pmatrix}
\hat U_{11}^\star & \hat U_{12}^\star - \mb E H \\
(\hat U^\star_{12})^\ast - \mb E H^\ast & \hat U_{22}^\star
\end{pmatrix} \r\|  \\
&\geq 
\l\|  \begin{pmatrix}
0 & \hat U_{12}^\star - \mb E H \\
(\hat U^\star_{12})^\ast - \mb E H^\ast & 0
\end{pmatrix} \r\| = 
\l\|  \hat U_{12}^\star - \mb E H \r\|,
\end{align*}
and the claim follows.

\subsection{Proof of Lemma \ref{lemma:variance}}
\label{proof:variance}

Recall that $\mu=\mb EY$. 

\noindent (a) Observe that 
\begin{align*}
\big\| \mb E\l(  (Y-\mu)(Y-\mu)^T \r)^2 \big\| &= 
\sup_{\|v\|_2=1} \mb E \l\langle v, Y-\mu \r\rangle^2 \l\| Y - \mu\r\|_2^2 \\
&
=\sup_{\|v\|_2=1}\l[ \sum_{j=1}^d \langle v, Y-\mu\rangle^2 (Y^{(j)} - \mu^{(j)})^2 \r].
\end{align*}
Next, for $j=1,\ldots,d$,
\begin{align*}
\mb E\langle v, Y-\mu\rangle^2 (Y^{(j)} - \mu^{(j)})^2 & \leq 
\mb E^{1/2} \langle v, Y-\mu\rangle^4 \, \mb E^{1/2}(Y^{(j)} - \mu^{(j)})^4  \\
&
\leq K \mb E \langle v, Y-\mu\rangle^2 \, \mb E (Y^{(j)} - \mu^{(j)})^2,
\end{align*}
hence 
\begin{align*}
\big\| \mb E\l(  (Y-\mu)(Y-\mu)^T \r)^2 \big\| \leq 
K\sup_{\|v\|_2=1}\mb E \langle v, Y-\mu\rangle^2 \sum_{j=1}^d \mb E (Y^{(j)} - \mu^{(j)})^2,
\end{align*}
and the result follows. 

\noindent (b) 
Note that
\begin{align*}
\tr\l[ \mb E\l(  (Y-\mu)(Y-\mu)^T \r)^2 \r] & =
\sum_{j=1}^d \mb E (Y^{(j)} - \mu^{(j)})^2 \l\| Y - \mu\r\|_2^2 \\
&
=\sum_{j=1}^d \mb E (Y^{(j)} - \mu^{(j)})^4 + \sum_{i\ne j} \mb E\l[ (Y^{(i)} - \mu^{(i)})^2 (Y^{(j)} - \mu^{(j)})^2 \r] \\
&
\leq \sum_{j=1}^d \mb E (Y^{(j)} - \mu^{(j)})^4 +  \sum_{i\ne j} \mb E^{1/2}(Y^{(i)} - \mu^{(i)})^4 \mb E^{1/2} (Y^{(j)} - \mu^{(i)})^4 \\
&
= \l( \sum_{j=1}^d \mb E^{1/2} (Y^{(j)} - \mu^{(j)})^4 \r)^2 
\leq K'\l( \sum_{j=1}^d  \mb E (Y^{(j)} - \mu^{(j)})^2\r)^2 \\
&
= K' \l( \tr(\Sigma) \r)^2.
\end{align*}

\noindent (c) The inequality follows from Corollary 5.1 in \cite{wei2017estimation}. 

\qed

\subsection{Proof of Corollary \ref{cor:frob}}
\label{proof:frob}

It is easy to see ((e.g., see the proof of Theorem 1 in \cite{lounici2014high}) that $\widetilde \Sigma_\star^\tau$ can be equivalently represented as 
\begin{align}
&
\widetilde \Sigma_\star^\tau=\argmin_{S\in \mb R^{d\times d}, S=S^T}\l[  \l\| S - \widetilde \Sigma_\star \r\|^2_{\mathrm{F}} +\tau \l\| S \r\|_1\r].
\end{align}
The remaining proof is based on the following lemma:
\begin{lemma}
Inequality (\ref{eq:ex70}) holds on the event $\m E=\l\{ \tau\geq 2\l\| \widetilde \Sigma_\star^\tau - \Sigma \r\| \r\}$. 
\end{lemma}
\noindent To verify this statement, it is enough to repeat the steps of the proof of Theorem 1 in \cite{lounici2014high}, replacing each occurrence of the sample covariance $\hat S_{2n}$ by its robust counterpart $\widetilde \Sigma_\star^\tau$. \\
Result of Corollary \ref{cor:frob} then follows from the combination of Theorem \ref{th:lepski} and Lemma \ref{lemma:variance} which imply that 
\[
\Pr(\m E)\geq 1-(4d+1)e^{-t}
\] 
whenever $\tau \geq \gamma\cdot 138\sqrt{K}\, \|\Sigma\| \,\sqrt{\frac{\mathrm{r}(\Sigma)(t+\Xi)}{\lfloor n/2\rfloor}}$.
\end{proof}

\section*{Acknowledgements}

Authors gratefully acknowledge support by the National Science Foundation grant DMS-1712956.

\bibliographystyle{imsart-nameyear}	
\bibliography{u-stat}

\end{document}

%% file: stas.tex
\newcommand{\dotp}[2]{\left\langle#1,#2\right\rangle}
\newcommand{\m}{\mathcal}
\newcommand{\mb}{\mathbb}
\newcommand\argmin{\mathop{\mbox{argmin}}}
\newcommand{\mx}{\mbox{\footnotesize{max}\,}}
\newcommand{\mn}{\mbox{\footnotesize{min}\,}}

\newcommand{\sign}{\mathrm{sign}}
\def\r{\right}
\def\l{\left}

\newcommand{\F}{\mathrm{F}}
\newcommand{\wh}{\widehat}

\newcommand{\pr}[1]{\mathrm{Pr}{\left(#1 \right)}}